\documentclass[11pt]{article}
\usepackage{mathrsfs}
\usepackage{amsfonts}
\usepackage{amssymb}
\usepackage{amsmath}
\usepackage{amsthm}
\usepackage{color}
\usepackage{enumitem}
\usepackage{tikz}

\usetikzlibrary{calc,positioning,patterns}

\newcommand{\circled}[1]{
    \tikz[baseline=(char.base)]{
        \node[shape=circle,draw,inner sep=1pt,fill=white] (char) {\scriptsize #1};
    }
}

\theoremstyle{plain}
\newtheorem{thm}{Theorem}[section]
\newtheorem{cor}[thm]{Corollary}
\newtheorem{lem}[thm]{Lemma}
\newtheorem{prop}[thm]{Proposition}

\newtheorem{Problem}{Problem}[section]
\newtheorem{Conj}{Conjecture}[section]
\numberwithin{equation}{section}

\newtheorem{Case}{Case}
\newtheorem*{Notation}{Notation}
\theoremstyle{definition}

\newtheorem{exam}{Example}[section]
\theoremstyle{remark}
\newtheorem{rem}{Remark}[section]

\newlength{\numwidth}
\newcommand{\formulaitem}[2][]{
  \refstepcounter{equation}
  \settowidth{\numwidth}{(\theequation)}
  \addtolength{\numwidth}{1em}          
  \begin{minipage}[b]{\dimexpr\linewidth-\numwidth\relax}
    #2
  \end{minipage}
  \hfill (\theequation)
  \ifx\relax#1\relax\else\label{#1}\fi%
}
\begin{document}
\title{{\Large\bf {Arveson’s Gauss-Bonnet-Chern Formula
for Hilbert Modules in the Multiplier-Algebra Framework\uppercase\expandafter{}}}
\thanks{
This work is supported by National Natural Science Foundation of China : (12271298 and 11871308).}}
\author{{\normalsize Penghui Wang, Ruoyu Zhang and Zeyou Zhu} \\
{ School of Mathematics, Shandong University,} {\normalsize Jinan, 250100, China} }
\date{}
 \maketitle
\begin{abstract}
The present paper continues Arveson's program on curvature, Euler characteristic, and Gauss-Bonnet-Chern type formulae for Hilbert modules. For regular unitarily invariant complete Nevanlinna--Pick spaces on the unit ball in \(\mathbb{C}^d\), we prove that Arveson's Gauss-Bonnet-Chern formula holds unconditionally over the full multiplier algebra. Together with the polynomial-ring case, this shows that the validity of the formula depends essentially on the coefficient algebra.
We then solve Arveson's finite defect problem in the same framework by giving a sharp classification of finite defect submodules. This problem is natural in curvature theory, since both curvature and Euler characteristic are governed by the defect structure. Finally, using commutative-algebra methods, we apply this classification to unitary-equivalence rigidity for finite-codimensional submodules of vector-valued reproducing kernel Hilbert modules. 
\end{abstract}

\bigskip


\bigskip

\noindent{$\mathbf{Key Words:}$ Hilbert modules; complete Nevanlinna-Pick spaces; multiplier algebras; Arveson's Gauss-Bonnet-Chern formula; finite defect problem.}

\section{Introduction}
Let $\mathbb T=(T_1,\cdots, T_d)$ be a tuple of commuting operators on a Hilbert space $\cal H$. Then $\cal H$ can be naturally viewed as a Hilbert module over the polynomial ring, equipped with the module action defined by 
$$
p\cdot x=p(T_1,\cdots,T_d) x,\,p\in\mathbb C[z_1,\cdots,z_d],\,x\in \cal{H}.
$$
A Hilbert module $\cal H$ is called a $d$-contractive Hilbert module if 
$
\sum\limits_{i=1}^d T_iT_i^*\leq I.
$  Historically, Drury-Arveson module $H_d^2$ serves as an important example of $d$-contractive Hilbert module with universal properties, which is a reproducing kernel Hilbert space of holomorphic functions on the unit ball $\mathbb B_d$ in $\mathbb C^d$, with the reproducing kernel
$$
{\mathcal K}(z,w)={\frac{1}{1-\langle z,
w\rangle}},\quad z,w\in\mathbb B_d,
$$   
and the module action is given by the natural multiplication by polynomials. Here, $H_d^2$ having the universal property means that any pure, finite rank $d$-contractive Hilbert module is unitarily equivalent to a quotient module in $H^2_d\otimes\mathbb C^N$ for some positive integer  $N$.

In the featured reviewed paper  \cite{Arveson curvature}, Arveson introduced the curvature invariant $K(\cal M^\perp)$ and Euler characteristic $\chi(\cal M^\perp)$ for quotient modules in $H_d^2\otimes \mathbb C^N(N\in \mathbb N^+)$. Moreover, for homogeneous quotient modules $\cal M^\perp$, Arveson \cite[Theorem B]{Arveson curvature} established an operator-theoretic analogue of the Gauss-Bonnet-Chern, abbreviated GBC formula
$$
K(\cal M^\perp)=\chi(\cal M^\perp),
$$
which will be called Arveson's GBC formula in what follows. Naturally, Arveson \cite{Arveson curvature} raised the following problem
\begin{Problem}\label{problem1}
 For which quotient modules of $H^2_d \otimes \mathbb{C}^N$
 does Arveson's GBC formula hold?   
\end{Problem}
In \cite{Fang Xiang}, Fang obtained Arveson's GBC formula for algebraic submodules, where a submodule is called algebraic if it can be generated by polynomials. It is worth pointing out that Arveson's GBC formula for noncommuting operator tuples has been extensively developed by Popescu \cite{Popescu2001, Popescu2015, Popescu2017} and the references therein.  Recently, by introducing the local algebraicity of submodules,  Problem \ref{problem1} was solved by the authors in \cite{curv}. 
\begin{thm}\label{them1}\cite[Theorem 1.1]{curv}
Let $\cal M$ be a submodule in $H^2_d\otimes \mathbb C^N$, then 
$$
K({\cal M}^\perp)=\chi({\cal M}^\perp)
$$
if and only if $\cal M$ is locally algebraic.
\end{thm}
Theorem \ref{them1} shows that Arveson's GBC formula holds true \emph{conditionally}. In the present paper, we will establish such a formula in a more general framework of a regular unitarily invariant complete Nevanlinna-Pick(CNP for short) reproducing kernel Hilbert space $H_{\mathbf K}$ with the reproducing kernel $\mathbf K:\mathbb B_d\times \mathbb B_d\to \mathbb C$. An $f\in H_{\mathbf{K}}$ is called a multiplier if the multiplication operator $M_{f}$ is continuous. The algebra of all multipliers is called the multiplier algebra, denoted by $\operatorname{Mult}(H_{\mathbf{K}})$, which has been extensively studied in \cite{CH,CT,HM} and the references therein. We will prove that Arveson's GBC formula holds true \emph{unconditionally} when we consider $H_{\mathbf K}$ as a Hilbert module over its multiplier algebra $\operatorname{Mult}(H_{\mathbf{K}})$. To illustrate this, we need some preparations.

\noindent{\bf 1.1 Regular unitarily invariant CNP kernel.} Let $\mathbf{K}: \mathbb{B}_d \times \mathbb{B}_d \rightarrow \mathbb{C}$  be a unitarily invariant reproducing kernel, defined by
\begin{equation}\label{k}
  \mathbf{K}(z,w)=\sum_{n=0}^{\infty}a_n \left\langle z,w \right\rangle^{n}, \, z,w\in \mathbb B_d,
\end{equation}
with $a_0 = 1$ and $a_n > 0$ for all $n \geq 1$. Furthermore, such a reproducing kernel is called  a regular  CNP kernel if
\begin{enumerate}[label=(\Roman*)] 
   \item it is an irreducible CNP kernel, or equivalently, there exists a sequence of non-negative real numbers $\{b_n\}_{n\geq 1}$ satisfying 
   \begin{equation}\label{1.a intro}
     1-\frac{1}{\mathbf{K}(z,w)}=\sum\limits_{n=1}^{\infty}b_n\left\langle z,w \right\rangle^{n},
   \end{equation}
      \item \formulaitem [asymtotic] {$\lim\limits_{n\rightarrow\infty}\frac{a_n}{a_{n+1}}=1$,}
      \item \formulaitem [sum bn] {$\sum\limits_{n=0}^{\infty}a_n=\infty$, i.e. $\sum\limits_{n=1}^{\infty}b_n=1$.} 
\end{enumerate}
In the past few decades, unitarily invariant CNP reproducing kernels have been systematically studied to attack interpolation problems; see \cite{AglerMcCarthy2000, AglerMcCarthyBook2002,DHS,HM,trent} and the references therein.

To continue, we need some notations. Let
\begin{equation}\label{N^d}
\mathbb{N}^{d} = \left\{ (\alpha_1,\cdots,\alpha_d) \mid \alpha_i \in \mathbb{N} \text{ for all } 1\leq i\leq d \right\},
\end{equation}
and $\mathbb{N}_+^{d}=\mathbb{N}^{d}\setminus \{(0,\cdots,0)\}$. Moreover, write 
$$
a_\alpha=\frac{|\alpha|!}{\alpha!}a_{|\alpha|} \text{ for $\alpha\in\mathbb{N}^{d}$ and }
 b_\beta=\frac{|\beta|!}{\beta!}b_{|\beta|}\text{ for }\beta\in\mathbb{N}_+^{d}.
$$ A $d$-tuple $\mathbb{T}=(T_1,\cdots,T_d)$  of commuting  bounded operators is called a $1/\mathbf{K}$-contraction if  $\sum\limits_{\alpha\in\mathbb{N}_+^{d}}b_\alpha\mathbb{T}^\alpha (\mathbb{T}^\alpha)^*$ converges in the strong operator topology with 
\begin{equation}\label{1/k contraction}
  I-\sum\limits_{\alpha\in\mathbb{N}_+^{d}}b_\alpha\mathbb{T}^\alpha (\mathbb{T}^\alpha)^*\geq 0,
\end{equation}
and in this case, we can define the defect operator $\Delta_{\mathbb{T},\mathbf{K}}$ for the tuple $\mathbb{T}$ to be  
\begin{eqnarray}\label{defect frac}
\Delta_{\mathbb{T},\mathbf{K}}=\left(1-\sum\limits_{\alpha\in\mathbb{N}_+^{d}}b_\alpha\mathbb{T}^\alpha (\mathbb{T}^\alpha)^*\right)^\frac{1}{2}.
\end{eqnarray}
{{A $1/\mathbf{K}$-contraction $\mathbb{T}$ is called pure if the series $\sum\limits_{\alpha\in\mathbb{N}^{d}}a_\alpha \mathbb{T}^\alpha \Delta_{\mathbb{T},\mathbf{K}}^2 (\mathbb{T}^\alpha)^*$ converges to $I$ in the strong operator topology.}}
Similar to the Drury-Arveson space, $H_{\mathbf{K}}$ has a universal property. Precisely, by \cite[Theorem 2.2]{CNP character}, any pure $1/\mathbf{K}$-contractive Hilbert module $\mathcal{H}$ is unitarily  equivalent to a quotient module $\mathcal{M}^\perp$ in $H_{\mathbf{K}} \otimes \overline{\operatorname{ran}\Delta_{\mathbb{T},\mathbf{K}}}$. 
  It follows that any pure $1/\mathbf{K}$-contraction $\mathbb{T}$ induces a module structure of the Hilbert space $\mathcal{H}$ over the multiplier algebra $\operatorname{Mult}(H_{\mathbf{K}})$. A pure $1/\mathbf{K}$-contractive Hilbert module $\mathcal{H}$ is called a finite rank Hilbert module if the defect operator $\Delta_{\mathbb{T},\mathbf{K}}$ is of finite rank. Without confusion, we will not distinguish between $\Delta_{\mathcal{H}}$ and $\Delta_{\mathbb{T}, \mathbf{K}}$. The rank of $\mathcal{H}$ is defined by $$\operatorname{rank}(\mathcal{H})= \dim\operatorname{ran}\Delta_{\mathcal{H}}.$$

\begin{rem}
It is worth pointing out that the regular unitarily invariant CNP reproducing kernel on the open unit ball $\mathbb B$ in an infinitely-dimensional separable Hilbert space $\cal H$  also makes sense. Precisely,  a unitarily invariant reproducing kernel $\mathbb{K}:\mathbb B\times \mathbb B\to \mathbb C$,  defined by
$$
 \mathbb{K}(z,w)=\sum_{n=0}^{\infty}a_n \left\langle z,w \right\rangle^{n}, \, z,w\in \mathbb B,
$$
is called a regular CNP reproducing kernel if the conditions (I)-(III) are satisfied, and the associated reproducing kernel Hilbert space is denoted by $H_{\mathbb K}$.  

\end{rem}
 
\noindent{\bf 1.2 Arveson's curvature invariant and Euler characteristic.} A closed subspace ${\cal M}\subset H_{\mathbf{K}}$ is called a submodule, if
$$p{\cal M}\subset{\cal M}, \text{ for }p\in {\cal P}_d,$$ and $\mathcal M^\perp$ is the associated quotient module. Now,  for a quotient module $\cal M^\perp$ in $H_{\mathbf{K}}\otimes\mathbb C^N$, let $S_{z_i}$ be the compression of $M_{z_i}\otimes I_{\mathbb{C}^N}$ onto the quotient space $\mathcal{M}^\perp$, i.e.
$$
S_{z_i} f=P_{\mathcal{M}^\perp}(M_{z_i}\otimes I_{\mathbb{C}^N}) f, \, \forall f\in \mathcal{M}^\perp.
$$
For $z\in\mathbb B_d$, define $S_{\mathbf{K},d}(z)=\sum\limits_{\alpha\in\mathbb{N}_+^{d}}b_\alpha\overline{z^\alpha} S_{z^\alpha}$, and
\begin{eqnarray}\label{eq:F(z)}
F(z)=\Delta_{\mathcal{M}^\perp} (I-S_{\mathbf{K},d}(z)^*)^{-1}(I-S_{\mathbf{K},d}(z))^{-1}\Delta_{\mathcal{M}^\perp},
\end{eqnarray}
which are positive bounded linear operators on  $\operatorname{ran} \Delta_{\mathcal M^\perp}$. It can be shown that the radial limits $\lim\limits_{r\to 1^-}\|\mathbf{K}_{r\xi}\|^{-2} \operatorname{tr}(F(r\xi))$ exists almost everywhere on the boundary $\partial \mathbb B_d$ relative to $d\sigma_d$ on $\partial \mathbb{B}_{d}$, and Arveson's curvature invariant is defined by
$$
K(\mathcal{M}^\perp)=\int_{\partial \mathbb B_d}\lim\limits_{r\to 1^-}\|\mathbf{K}_{r\xi}\|^{-2} \operatorname{tr}(F(r\xi))d\sigma_d(\xi),
$$
where $\sigma_d$ is the rotationally invariant probability measure on $\partial \mathbb{B}_{d}$. For quotient modules in Drury-Arveson module, such an invariant was deeply studied by \cite{Arveson Dirac, CNP Curvature, CNP character, Fang Xiang, GRS, inner multipliers, Levy, Muhly, Parrott} and the references therein.

In \cite{Arveson curvature}, Arveson introduced the Euler characteristic for a quotient module $\cal{M}^\perp$ in $H_{d}^2\otimes\mathbb C^N$. 
To simplify the notation, in what follows, set ${\cal P}_d=\mathbb C[z_1,\cdots,z_d]$. Let 
\begin{eqnarray}\label{eq:fgmodule}
\mathbb{M}_{\mathcal{M}^\perp}=\operatorname{span}\left\{f(S_{z_1},\cdots,S_{z_d}) \cdot \Delta_{\mathcal{M}^\perp} \zeta \mid f \in {\cal P}_d, \zeta \in \mathcal{M}^\perp\right\}.
\end{eqnarray}
Since ${\cal P}_d$ is Noetherian,  $\mathbb{M}_{\mathcal{M}^\perp}$ has finite free resolutions
 \begin{eqnarray}\label{Euler-ch}
 0\rightarrow F_n\rightarrow\cdots\rightarrow F_2\rightarrow F_1\rightarrow \mathbb{M}_{\mathcal{M}^\perp,{\mathcal P}_d}\rightarrow 0,\end{eqnarray}
 each $F_k$ being a sum of $\beta_k$ copies of the free module ${\cal P}_d$. The Euler characteristic is defined by
\begin{equation}\label{chenlaoshi}
 \chi_{\mathcal P_d}(\mathcal{M}^\perp)=\sum\limits_{k=1}^{n}(-1)^{k+1}\beta_k.
\end{equation}
Let 
$\mathcal{R}\subset\operatorname{Mult}(H_\mathbf{K})$ 
be a unital subalgebra. For example, we may let 
$$
{\mathcal R}={\mathcal P}_d,\,H(\overline{\mathbb{B}}_d),\, A(H_{\mathbf K})\text{ or } \operatorname{Mult}(H_{\mathbf{K}}),
$$ 
where $H(\overline{\mathbb{B}}_d)$ is the algebra of holomorphic functions on some neighbourhoods of the closed unit ball $\overline{\mathbb{B}}_d$ and $A(H_{\mathbf{K}})$ denotes the closure of ${\cal P}_d$ in $\operatorname{Mult}(H_{\mathbf{K}})$. Similar to (\ref{eq:fgmodule}), for any quotient module ${\cal M}^\perp$ in $H_{\mathbf{K}}$ one can naturally define
\begin{eqnarray}
\mathbb{M}_{\mathcal{M}^\perp,{\mathcal R}}=\operatorname{span}\left\{f(S_{z_1},\cdots,S_{z_d}) \cdot \Delta_{\mathcal{M}^\perp} \zeta \right|\left. f \in \mathcal{R}, \zeta \in \mathcal{M}^\perp\right\}.
\end{eqnarray}
Notice that the algebra $\mathcal R$ may not be Noetherian, and the finite free resolution (\ref{Euler-ch}) may not exist and hence the Euler characteristic can not be defined directly. However, for $\mathcal R={\mathcal P_d}$ it can be easily verified that $\chi_{{\mathcal P}_d}(\mathcal{M}^\perp)={\operatorname {rank}} ({\mathbb M_{{\cal M}^\perp,\mathcal P_d}})$, where ${\operatorname{rank}} ({\mathbb M_{{\cal M}^\perp,\mathcal P_d}})$ is the rank of the finitely generated module $\mathbb M_{{\cal M}^\perp,\mathcal P_d}$; see \cite[Page 873]{Rotman}.
 {\emph{In what follows, for convenience we denote }
\begin{equation}\label{oulashudedingyi}
 \chi_{\mathcal R}(\cal{M}^\perp)=\operatorname{rank} ({\mathbb M_{\cal{M}^\perp,\mathcal R}}).
\end{equation}
Based on  this observation, Problem \ref{problem1} for Arveson's GBC formula for Hilbert modules over $\mathcal R$ should be revised as the following problem. }
\begin{Problem}\label{Problem 2}
 For which  submodules $\cal M$ in $H_{\mathbf{K}} \otimes \mathbb{C}^N$, 
 $$
 K(\cal{M}^\perp)=\chi_{\cal R}(\cal{M}^\perp)?
 $$
\end{Problem}  

\noindent{\bf 1.3 Main results.} There are two main contributions in this paper. In the settings of vector-valued regular unitarily invariant CNP reproducing kernel Hilbert modules, firstly we establish Arveson's GBC formula; secondly we solve the finite defect problem which is used to attack the unitary-equivalence rigidity problem.

\noindent{\bf 1.3.1 Arveson's GBC formula.} First, we illustrate our main result on Arveson's GBC formula as follows. 
The following theorem states that whether Arveson's GBC formula holds true unconditionally depends heavily on the choice of the unital subalgebra of $\operatorname{Mult}(H_{\bf K})$.
\begin{thm}\label{thm1.4}
Let $\cal R$ be a unital subalgebra of $\operatorname{Mult}(H_{\bf K})$. Then we have
\begin{itemize}
\item[1)] if $\cal R$ contains a nonzero ideal of $\operatorname{Mult}(H_{\mathbf{K}})$, then Arveson's GBC formula always holds true unconditionally, i.e. for any nonzero submodule $\cal M$ in $H_{\mathbf{K}}\otimes\mathbb{C}^N$,
$$
K({\cal M}^\perp)=\chi_{\cal R}(\cal M^{\perp}),
$$
\item[2)] if $\cal R$ is a unital subalgebra of $A(H_{\mathbf{K}})$, then for any $0\leq m\leq N$, there is a submodule $\cal M$ in $H_{{\bf K}}\otimes \mathbb C^N$, such that 
$$
\chi_{\cal R}({\cal M}^{\perp})-K({\cal M}^\perp)=m.
$$
\end{itemize}
\end{thm}

The following figure illustrates the two cases of coefficient subalgebras in Theorem \ref{thm1.4}.   

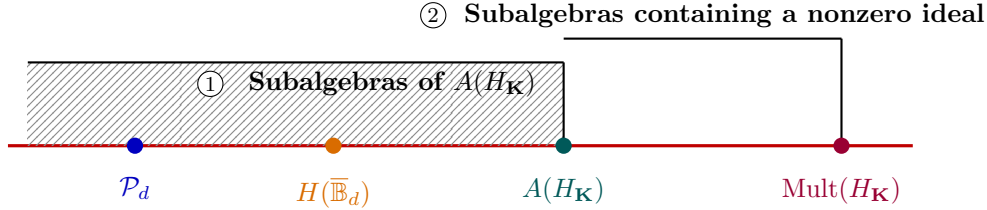
\begin{figure}[htbp]
\centering
\begin{tikzpicture}[scale=1.05, every node/.style={font=\small}]

    \coordinate (P1) at (1.6,0);    
    \coordinate (P2) at (4.1,0);    
    \coordinate (P3) at (7.0,0);    
    \coordinate (P4) at (10.5,0);   

    \coordinate (L0) at (0.25,0);   
    \coordinate (L1) at (0.25,1.05);
    \coordinate (R1) at (7.0,1.05); 

    \draw[red!75!black, very thick] (0,0) -- (11.4,0);

    \path[
        draw=none,
        fill=blue!7,
        pattern=north east lines,
        pattern color=black!45
    ]
        (L0) -- (L1) -- (R1) -- (P3) -- cycle;

    \draw[black, thick] (L1) -- (R1);
    \draw[black, thick] (R1) -- (P3);

    \node at (2.55,0.78) {\circled{1}};
    \node[anchor=west] at (2.90,0.78) {\textbf{Subalgebras of \(A(H_{\mathbf{K}})\)}};

    \draw[black, thick] (P4) -- ++(0,1.35);
    \draw[black, thick] ($(P4)+(0,1.35)$) -- ++(-3.5,0);
    \node[above] at ($(P4)+(-1.75,1.35)$)
    {\circled{2}\ \textbf{Subalgebras containing a nonzero ideal}};

    \fill[blue!75!black]   (P1) circle (2.9pt);
    \fill[orange!85!black] (P2) circle (2.9pt);
    \fill[teal!70!black]   (P3) circle (2.9pt);
    \fill[purple!80!black] (P4) circle (2.9pt);

   \node[below=8pt, text=blue!75!black]   at (P1) {$\mathcal P_d$};
    \node[below=8pt, text=orange!85!black] at (P2) {$H(\overline{\mathbb{B}}_d)$};
    \node[below=8pt, text=teal!70!black]   at (P3) {$A(H_{\mathbf{K}})$};
    \node[below=8pt, text=purple!80!black] at (P4) {$\operatorname{Mult}(H_{\mathbf{K}})$};

\end{tikzpicture}
\label{fig:ring-division}
\end{figure}

It will be seen that, in analogy with the above results for reproducing kernel Hilbert space over $\mathbb B_d$, there is no essential difficulty in obtaining Arveson's GBC formula for regular unitarily invariant CNP reproducing kernel Hilbert space s whose kernels are defined on the open unit ball of a separable infinite-dimensional Hilbert space.

\noindent{\bf 1.3.2 Finite defect problem.} It is easy to see that any quotient module in $H_{\mathbf K }\otimes\mathbb C^N$ has finite rank. A natural problem, raised by Arveson \cite[Page 226]{Arveson curvature} for submodules in $H_d^2$, is to study when the submodule in $H_{d}^2$ is of finite rank. Such a problem is called \emph{the finite defect problem}. 
  \begin{Problem}\label{problem3}
   Is the rank of each nonzero submodule of $H_d^2$ that has infinite codimension in $H_d^2$
  infinite?
\end{Problem}
By providing a positive answer, problem \ref{problem3} was solved by Guo \cite{Gu}.
Notice that the finite defect problem also makes sense for general reproducing kernel Hilbert spaces.  The following result shows that the finite defect problem depends on the reproducing kernel. 
\begin{thm}\label{thm1.5}
  Let $\mathbf K$ ($\mathbb K$ for infinitely-many variable case) be a regular unitarily invariant CNP reproducing kernel on the open unit ball in a Hilbert space $\cal H$, and $\mathcal{M}$ be a nonzero submodule in $H_{\mathbf{K}}\otimes \mathbb C^N$ (or $H_{\mathbb K}\otimes \mathbb C^N$). Then
\begin{itemize}
\item[(1)] for $2\leq \dim{{\cal H}}<\infty$,
  \begin{itemize}
  \item[1.] in the case $\{n: b_n>0\}$ is a finite set, $\mathcal{M}$ is of finite rank if and only if there is a subspace $F$ in $\mathbb C^N$, such that  $\cal M$ is of finite codimension in $H_{\mathbf K}\otimes F$;
  \item[2.] in the case $\{n: b_n>0\}$ is an infinite set, $\mathcal{M}$ is of finite rank if and only if there is a subspace $F$ in $\mathbb C^N$ such that ${\cal M}=H_{\mathbf{K}}\otimes F$;
\end{itemize}
\item[(2)] for $\dim {\cal H}=\infty$, $\Delta_{\cal M}$ is compact if and only if there is a subspace $F$ in $\mathbb C^N$ such that ${\cal M}=H_{\mathbb{K}}\otimes F$, which is equivalent to $\Delta_{\cal M}$ being of finite rank. 
\end{itemize}
\end{thm}
As a special case of Theorem \ref{thm1.5}, we have the following corollary.
\begin{cor}\label{cor1}
Let $\cal D$ be the Dirichlet module, then the only  nonzero submodule of Dirichlet module $\cal D$ with finite defect is the whole module $\cal D$. 
\end{cor}

\begin{rem}
\begin{itemize}
\item[1)] Guo proved that, in the scalar-valued case, a submodule $\mathcal M$ of the Drury--Arveson module has finite defect if and only if it has finite codimension \cite{Gu}, and that the same equivalence holds for polynomially generated submodules of the Hardy and weighted Bergman modules \cite{Gu2}. Comparing these results with  Corollary \ref{cor1}, even in the scalar-valued case, the finite defect problem for submodules in Dirichlet module is totally different. 
\item[2)] By the above theorem, in the infinitely-many-variable settings,  there is no nonzero proper submodule $\cal M$ in $H_{\mathbb K}$ such that $\Delta_{\cal M}$ is Schatten-$p$ class, which is totally different from the finitely-many-variable case. In fact, most known results show that $\Delta_{\cal M}$ belongs to Schatten-$p$ class for some $p>d$, details can be seen in the survey paper \cite{GWy}. 
\end{itemize}
\end{rem}

In Section 4, we will consider the rigidity problem of finite-codimensional submodules in vector-valued reproducing kernel Hilbert modules, which was solved in \cite{GC} for scalar-valued case by using Characteristic Space Theory developed in \cite{Guo1}. However, for the vector-valued case, Characteristic Space Theory does not work, and the problem for vector-valued case will be solved using commutative-algebraic methods. 

For $d\ge 2,$ Let $ \Omega \subset \mathbb{C}^d $ be a bounded domain, and let
$H_{\mathfrak{K}}$ be a reproducing kernel Hilbert space of holomorphic functions on
$ \Omega $. Assume that \(\mathcal P_d\) is dense in
\(H_{\mathfrak K}\), that the coordinate multipliers
\(M_{z_1},\ldots,M_{z_d}\) are bounded on \(H_{\mathfrak K}\), and that
$
        \mathfrak K(\lambda,\lambda)\longrightarrow\infty,
        \lambda\to\partial\Omega .
$
Then we have the following theorem.
\begin{thm}\label{thm:rigidity} Under the above assumptions, 
let \(\mathcal{M}_1,\mathcal{M}_2\subset H_{\mathfrak{K}}\otimes \mathbb C^N\) be finite-codimensional submodules. Suppose there exists
a unitary module isomorphism
$$
U:\mathcal{M}_1\to \mathcal{M}_2
$$
i.e. $U$ is unitary and
$$
U\bigl(M_{z_j}\otimes I|_{\mathcal{M}_1}\bigr)
=
\bigl(M_{z_j}\otimes I|_{\mathcal{M}_2}\bigr)U,
\quad j=1,\dots,d.
$$
Then there exists a constant unitary matrix $W\in M_N(\mathbb{C})$ such that $\mathcal{M}_2=(I_{H_{\mathfrak{K}}}\otimes W)\mathcal{M}_1$. In particular, for \(N=1\), $\mathcal{M}_1$ is unitarily equivalent to $\mathcal{M}_2$ if and only if $\mathcal{M}_1=\mathcal{M}_2$.
\end{thm}
By Theorem \ref{thm1.5}, in the case that $\{n, b_n>0\}$ is finite,  if ${\cal M}_1$ and ${\cal M}_2$ are unitarily equivalent, then ${\cal M}_1$ is of finite codimension if and only if so is ${\cal M}_2$. It follows that in this case, to get the rigidity result as in Theorem \ref{thm:rigidity},  we may only assume one of ${\cal M}_i$, $i=1,2$  being of finite codimension.

This paper is organized as follows.
Section 2 introduces the curvature invariants and the Euler characteristic relative to a unital subalgebra of the multiplier algebra. Using a local characterization of the Euler characteristic, we establish the GBC formula in the multiplier-algebra framework and prove the main result, Theorem \ref{thm1.4}.
 Section 3 gives a complete solution to the finite defect problem in the case of finitely many variables for regular unitarily invariant CNP reproducing kernel Hilbert spaces, proving Theorem \ref{thm1.5}(1). Section 4 uses techniques from commutative algebra to study the unitary-equivalence rigidity of finite-codimensional submodules of vector-valued reproducing kernel Hilbert spaces over bounded domains in $\mathbb{C}^d$, leading to Theorem \ref{thm:rigidity}. Section 5 gives a complete characterization of the finite defect problem in the infinitely-many-variable case for regular unitarily invariant CNP reproducing kernel Hilbert spaces, proving Theorem \ref{thm1.5}(2). Finally, Section 6 establishes a negative result for the geometric Arveson-Douglas conjecture in the infinitely-many-variable setting.

\noindent

\section{Arveson's GBC formula over the multiplier algebra}
In this section, we elaborate on the curvature invariant for $1/\mathbf{K}$-contractive Hilbert modules and prove Theorem \ref{thm1.4}. First, we need some preparations. The following proposition is taken from \cite[Corollary 4.2]{CNP character}, and we include it here for the reader’s convenience.
\begin{prop}\cite[Corollary 4.2]{CNP character}
  Let $\mathbb{T}$ be a $1/\mathbf{K}$-contraction on a Hilbert space $\mathcal{H}$. Then there is a unique contractive linear operator
 $$
 L: H_{\mathbf{K}}\otimes \overline{\operatorname{ran}\Delta_\mathcal{H}}\rightarrow \mathcal{H}
 $$
 satisfying
  $$
  L(p\otimes \zeta)= p(\mathbb{T})\Delta_\mathcal{H}\zeta,~\forall p\in \mathcal{P}_d,~\zeta \in \overline{\operatorname{ran}\Delta_\mathcal{H}}.
  $$
In particular, if $\mathbb{T}$ is pure, then $L$ is a coisometry.
\end{prop}
   For a $1/\mathbf{K}$-contraction $\mathbb{T}$ on $\mathcal{H}$, define an operator 
   \begin{equation}\label{T_kd(z)}
   \mathbb{T}_{\mathbf{K}}(z)=\sum\limits_{\alpha\in\mathbb{N}_+^{d}}b_\alpha\bar{z}^\alpha T^\alpha,\, z\in \mathbb{B}_d.
   \end{equation}
  By (\ref{1/k contraction}), for $\zeta\in\mathcal{H}$, 
   $$
   \|\mathbb{T}_{\mathbf{K}}(z)\zeta\|^2=\left\|\sum\limits_{\alpha\in\mathbb{N}_+^{d}}b_\alpha\bar{z}^\alpha T_\alpha\zeta\right\|^2\leq \sum\limits_{\alpha\in\mathbb{N}_+^{d}}\left\|\sqrt{b_\alpha}\bar{z}^\alpha \zeta\right\|^2=\left(\sum\limits_{\alpha\in\mathbb{N}_+^{d}}b_\alpha|z^\alpha|^2\right)\left\| \zeta\right\|^2.
   $$
It follows that for $z\in\mathbb B_d$, $$\|\mathbb{T}_{\mathbf{K}}(z)\|^2\leq\sum\limits_{\alpha\in\mathbb{N}_+^{d}}b_\alpha|z^\alpha|^2\leq 1-{\frac{1}{K(z,z)}}<1,$$
 which implies that  $I-\mathbb{T}_{\mathbf{K}}(z)$ is invertible. Similar to (\ref{eq:F(z)}), define $F:\mathbb{B}_d\rightarrow B\left(\overline{\operatorname{ran}\Delta_\mathcal{H}}\right)$ by
$$
F(z)=\Delta_\mathcal{H}(I-\mathbb{T}_{\mathbf{K}}(z)^*)^{-1}(I-\mathbb{T}_{\mathbf{K}}(z))^{-1} \Delta_\mathcal{H}.
$$
By \cite[Theorem 4.11]{CNP character}, the same reasoning as \cite[Theorem 1.2 and Theorem A]{Arveson curvature} gives the following proposition and we omit the proof.
\begin{prop}
Let $\mathbb{T}$ be a $1/\mathbf{K}$-contraction of $\mathcal{H}$ of finite defect. Then there is a Hilbert space $E$ and an operator-valued holomorphic function $\Phi:\mathbb{B}_d\rightarrow B(E,\overline{\operatorname{ran}\Delta_\mathcal{H}})$ such that
  $$
  \|\mathbf{K}_z\|^{-2}F(z)=I-\Phi(z)\Phi(z)^*,~\forall z\in \mathbb{B}_d.
  $$  
Moreover, the radial limits $\lim\limits_{r\to 1^-}\|\mathbf{K}_{r\xi}\|^{-2} \operatorname{tr}(F(r\xi))$ exist almost everywhere on the boundary $\partial \mathbb B_d$ relative to $d\sigma_d$ on $\partial \mathbb{B}_{d}$.
\end{prop}
  Set $K_0(z)=\lim\limits_{r\to 1^-}\|\mathbf{K}_{rz}\|^{-2} \operatorname{tr}(F(rz))$ for $\sigma_d~ a.e. z\in \partial\mathbb{B}_d$. The curvature invariant of module $\mathcal{H}$ is defined by
$$
K(\mathcal{H})=\int_{\partial \mathbb B_d}K_0(z)d\sigma_d(z).
$$
By \cite[Theorem 4.11]{CNP Curvature}, 
\begin{equation}\label{2.a}
  K(\mathcal{H})=\operatorname{rank}\mathcal{H}-\operatorname{fd}(\operatorname{ran}\Phi(\lambda)),
\end{equation}
where for any subspace $V$ in $H_{\mathbf{K}}\otimes \mathbb C^N$, 
$$
\operatorname{fd}(V)=\sup\limits_{\lambda\in\mathbb B_d} \mathrm{dim} E_\lambda(V).
$$
Set 
$$
 \text{mp}(V)=\{\lambda \in \mathbb{B}_d: \operatorname{dim}E_{\lambda}V=\operatorname{fd}(V)\}.
$$
Now, for a submodule $\cal M$  in $H_{\mathbf{K}}\otimes \mathbb C^N$ set 
$$
n=\mathrm{dim}({\cal M}\cap (\mathbb C\otimes \mathbb C^N)),
$$
and choose an orthonormal basis $\{1\otimes e_i\}_{i=1}^{N-n}$ of  $\mathbb C\otimes\mathbb C^N\ominus({\cal M}\cap (\mathbb C\otimes\mathbb C^N))$.  By \cite[Lemma 2.1]{CNP character}, $I-\sum\limits_{\alpha\in \mathbb{N}_+^{d}}^{\infty}b_\alpha M_{z^\alpha}M_{z^\alpha}^*$ is the orthogonal projection from $H_{\mathbf{K}}$ onto the one-dimensional subspace of constant functions. Then applying the same reasoning as in \cite[Lemma 2.1 and Lemma 2.2]{curv} we have the following lemma.
\begin{lem}\label{dimdefect} Under the above assumptions,  
$\{P_{\mathcal{M}^\perp}(1 \otimes e_i)\}_{i=1}^{N-n} \text{ is a linear basis of } \operatorname{ran} \Delta_{\mathcal{M}^\perp}$, 
and
\begin{equation}\label{K=N-fd}
K(\mathcal{M}^\perp) = N - \operatorname{fd}(\mathcal{M}).
\end{equation}
\end{lem}
To distinguish between the linear independence of a subset in a module and the linear independence of a subset in a vector space, we provide the following notation. 
\begin{Notation}
Let $\mathcal{A}$ be a module over an integral domain $\mathfrak{R},$ a finite subset $\{a_i\}_{i=1}^m\subseteq \mathcal{A}$ is called linearly independent {\it in the module $\mathcal{A},$} if
$$\sum\limits_{i=1}^mp_ia_i=0, \quad p_i \in \mathfrak{R}$$ implies that $p_i=0.$ A subset $\mathfrak{S}$ is a basis of the module $\mathcal{A}$ if $\mathfrak{S}$ is linearly independent in the module $\mathcal{A},$ and generates $\mathcal{A}.$
\end{Notation}
Let $F$ be a finitely generated module over the integral domain $\mathfrak{ R}$, let $\mathfrak{S}$ be a generating set of $F$. It is an exercise that the number of elements in any maximal linearly independent subset of $\mathfrak{S}$ in the module F is $\operatorname{rank}(F)$, regardless of the choice of $\mathfrak{S}$. The key to prove Theorem \ref{thm1.4} is the technique of characterizing the Euler characteristic from a local perspective. We have the following theorem.
\begin{thm}\label{thm 1.2} 
Let $\cal R$ be any unital subalgebra of $\operatorname{Mult}(H_{\mathbf{K}})$.
Then for any submodule $\cal M$ in $H_{\mathbf{K}}\otimes \mathbb{C}^N$, the following statements are equivalent
\begin{itemize}
     \item[(1)] $\chi_{\cal R}(\mathcal{M}^\perp)=l$,
     \item[(2)] there is an open set $U$ that is dense in $\mathbb{B}_d$, such that for all $\lambda\in U$,
  $$
  \dim \{E_\lambda(\mathcal{M}\cap ({\cal R}\otimes \mathbb{C}^N))\}=N-l,
  $$
  where for any subspace $V$ in $H_{\mathbf{K}}\otimes \mathbb{C}^N$,  
  $$
  E_\lambda(V)=\{f(\lambda): f\in V\}.
  $$
\end{itemize}
\end{thm}
\begin{proof}
  $(1)\Rightarrow(2)$: Set $$n=\dim \mathcal{M}\cap(\mathbb{C}\otimes\mathbb{C}^N).$$ By Lemma \ref{dimdefect}, $\{P_{\mathcal{M}^\perp}(1\otimes e_i)\}_{i=1}^{N-n}$ generates $\mathbb{M}_{\mathcal{M}^\perp, \mathcal{R}}$. Since 
  $$
  \chi_\mathcal{R}(\mathcal{M}^\perp)=\operatorname{rank}(\mathbb M_{\cal{M}^\perp, {\cal R}})=l,
  $$
the number of elements in any maximal linearly independent subset of $\{P_{{\cal M}^\perp}(1\otimes e_i)\}_{i=1}^{N-n}$ in the module $\mathbb{M}_{{\cal M}^\perp, \mathcal{R}}$ is $l$. Without loss of generality, assume that $\{P_{\mathcal{M}^\perp}(1\otimes e_i)\}_{i=1}^l$ is a maximum linearly independent set of $\{P_{\mathcal{M}^\perp}(1\otimes e_i)\}_{i=1}^{N-n}$ in the module $\mathbb{M}_{ \mathcal{M}^\perp,\mathcal{R}}$. It follows from 
$$P_{\mathcal{M}^\perp}(1\otimes e_i)=0, \text{ for }N-n+1\leq i\leq N$$ 
that $\{P_{\mathcal{M}^\perp}(1\otimes e_i)\}_{i=1}^{l}$ is a maximum linearly independent set of $\{P_{\mathcal{M}^\perp}(1\otimes e_i)\}_{i=1}^{N}$  in the module $\mathbb{M}_{\mathcal{M}^\perp, \mathcal{R}}$. Then for each $l+1\leq i\leq N$, there exist $p_{1,i},\cdots,p_{l,i},p_{l+1,i}\in \mathcal{R}$ which are not all zero, such that
  $$
  q_i=p_{1,i}\otimes e_1 +\cdots +p_{l,i}\otimes e_l +p_{l+1,i}\otimes e_i\in \mathcal{M}.
  $$
 Since $\{P_{\mathcal{M}^\perp}(1\otimes e_i)\}_{i=1}^{l}$ is linearly independent in the module $\mathbb{M}_{ \mathcal{M}^\perp,\mathcal{R}}$, we have
  $$p_{l+1,i}\neq 0, \text{ for } l+1\leq i\leq N.$$ 
  Set
  $$
  U=\mathbb{B}_d\setminus\mathop{\cup}\limits_{i=l+1}^{N} Z(p_{l+1,i}),
  $$
  which is dense in $\mathbb{B}_d$. And for all $\lambda\in U$, $q_{l+1}(\lambda),\cdots,q_N(\lambda)$ is linearly independent. Hence 
  $$
  \dim \{E_\lambda(\mathcal{M}\cap (\mathcal{R}\otimes \mathbb{C}^N))\}\geq N-l,~\forall \lambda\in U.
  $$
  Next we claim that $\dim \{E_\lambda(\mathcal{M}\cap (\mathcal{R}\otimes \mathbb{C}^N))\}\leq N-l$ for all $\lambda\in U$. Otherwise there is a $\lambda_0\in U$ such that
  $$
  \dim \{E_{\lambda_0}({\cal M}\cap (\mathcal{R}\otimes \mathbb{C}^N))\}\geq N-l+1.
  $$
  Then there exist $h_1,\cdots,h_{N-l+1}\in \mathcal{M}\cap (\mathcal{R} \otimes \mathbb{C}^N)$ such that 
  $$
  \dim \operatorname{span}\{h_1(\lambda_0),\cdots,h_{N-l+1}(\lambda_0)\}=N-l+1.
  $$
  Write $h_i=(h_{i,1},\cdots,h_{i,N})$ for $1\leq i \leq N-l+1$. Hence
   \begin{eqnarray}\label{eq:2}
\operatorname{rank} \left[
                 \begin{array}{cccc}
                   h_{1,1}(\lambda_{0}) & h_{1,2}(\lambda_{0}) & \cdots & h_{1,N}(\lambda_{0}) \\
                   h_{2,1}(\lambda_{0}) &h_{2,2}(\lambda_{0}) & \cdots &h_{2,N}(\lambda_{0}) \\
                   \vdots & \vdots & \ddots & \vdots \\
                    h_{N-l+1,1}(\lambda_{0}) & h_{N-l+1,2}(\lambda_{0}) & \cdots & h_{N-l+1,N}(\lambda_{0})\\
                 \end{array}
               \right]=N-l+1.
\end{eqnarray}
Thus there exist $1\leq i_1<\cdots<i_{N-l+1}\leq N$ such that the determinant
  \begin{eqnarray}\label{eq:3}
 \left|
                 \begin{array}{cccc}
                   h_{1,i_1}(\lambda_{0}) &h_{1,i_2}(\lambda_{0}) & \cdots & h_{1,i_{N-l+1}}(\lambda_{0}) \\
                   h_{2,i_1}(\lambda_{0}) &h_{2,i_2}(\lambda_{0}) & \cdots &h_{2,i_{N-l+1}}(\lambda_{0}) \\
                   \vdots & \vdots & \ddots & \vdots \\
                    h_{N-l+1,i_{1}}(\lambda_{0}) & h_{N-l+1,i_2}(\lambda_{0}) & \cdots & h_{N-l+1,i_{N-l+1}}(\lambda_{0})\\
                 \end{array}
               \right|\neq 0.
\end{eqnarray}
 which implies that the determinant
  \begin{eqnarray}\label{eq:4}
 \left|
                 \begin{array}{cccc}
                   h_{1,i_1} & h_{1,i_2} & \cdots & h_{1,i_{N-l+1}} \\
                   h_{2,i_1} &h_{2,i_2} & \cdots &h_{2,i_{N-l+1}} \\
                   \vdots & \vdots & \ddots & \vdots \\
                    h_{{N-l+1},i_1} & h_{{N-l+1},i_2} & \cdots & h_{{N-l+1},i_{N-l+1}}\\
                 \end{array}
               \right|\neq 0.
\end{eqnarray}
We claim that $\operatorname{rank}\mathbb{M}_{\mathcal{M}^\perp,{\cal R}}\leq l-1$, which contradicts that 
$$
\operatorname{rank}\mathbb{M}_{\mathcal{M}^\perp,{\cal R}}=\chi_{\cal R}(\mathcal{M}^\perp)=l.
$$
In fact, by Lemma \ref{dimdefect}, it suffices to show that
for any $1\leq m_1<\cdots<m_{l}\leq N,$  $\{{P_{\mathcal{M}^\bot}(1\otimes e_{m_i})}\}_{i=1}^{l}$ is
  linearly dependent in the module $\mathbb{M}_{\mathcal{M}^\bot,{\cal R}}$. Set
$$
\{v_1,\cdots,v_{N-l}\}=\{1,\cdots,N\}\setminus \{m_1,\cdots,m_{l}\},
$$ 
and
 \begin{eqnarray}\label{eq:6}
 C^{(k)}=\left[
                 \begin{array}{cccc}
                   h_{1,k} & h_{1,v_1} & \cdots & h_{1,v_{N-l}}\\
                   h_{2,k} &h_{2,v_1} & \cdots &h_{2,v_{N-l}} \\
                   \vdots & \vdots & \ddots & \vdots \\
                    h_{{N-l+1},k} & h_{{N-l+1},v_1} & \cdots & h_{{N-l+1},v_{N-l}}\\
                 \end{array}
               \right], \quad 1\leq k\leq N.
\end{eqnarray}
Let $C^{(k)}_{ij}$ be the algebraic cofactor of the $(i,j)$ element in $C^{(k)}.$
Notice that for all $1\leq i \leq N-l+1,$ $$C^{(1)}_{i1}=C^{(2)}_{i1}=\cdots=C^{(N)}_{i1},$$ and write
 $r_{i1}=C^{(k)}_{i1}.$ Next, we claim that  $\{{P_{\mathcal{M}^\bot}(1\otimes e_{m_i})}\}_{i=1}^{l}$ is
  linearly dependent in the module $\mathbb{M}_{\mathcal{M}^\bot,\cal R} $, which will be proved in two cases.
  \begin{Case}
There exists an $i_0 \in \{1,\cdots,N-l+1\}$ such that $r_{i_01}=0.$
\end{Case}
There are $f_1, \cdots, f_{i_0-1}, f_{i_0+1}, \cdots, f_{N-l+1}\in \cal R$, which are not all zero, such that
 \begin{eqnarray}\label{eq:10}
\left(\sum\limits_{i\neq i_0} f_i h_{i,v_1},\cdots,\sum\limits_{i\neq i_0} f_i h_{i,v_{N-l}}\right)=\sum\limits_{i\neq i_0} f_i \left(h_{i,v_1}, \ldots, h_{i,v_{N-l}}\right)=0.
\end{eqnarray}
By \eqref{eq:4}, $\{h_1, \cdots, h_{i_0-1}, h_{i_0+1}, \cdots, h_{N-l+1}\}$ are linearly independent in the module ${\cal R}\otimes \mathbb{C}^N $, which gives that
 \begin{eqnarray}\label{eq:11}
\sum\limits_{1\leq i\leq N-l+1,i\neq i_0} f_i h_i \neq 0.
\end{eqnarray}
For simplicity, Write $\left(g_1, \cdots,g_j,\cdots, g_N\right)=\sum\limits_{i\neq i_0} f_i h_i.$
 By \eqref{eq:10},
$
g_{v_i}=0 ~\text{for all} ~1 \leq i \leq N-l,
$
and
hence by \eqref{eq:11}, $(g_{m_1},\cdots,g_{m_{l}})\neq0.$ Then
$$
\sum\limits_{i=1}^{l} g_{m_i}\otimes e_{m_i}=\sum\limits_{i=1}^Ng_i\otimes e_i=\sum\limits_{i\neq i_0} f_i h_i\in \mathcal{M}.
$$
 This ensures that
 $$\sum\limits_{i=1}^{l}g_{m_i}\cdot{P_{\mathcal{M}^\bot}(1\otimes e_{m_i})}={P_{\mathcal{M}^\bot}\left(\sum\limits_{i=1}^{l}g_{m_i}\otimes e_{m_i}\right)}=0.$$
 \begin{Case}
 For all $1 \leq i \leq N-l+1,$ $r_{i1}\neq 0.$
  \end{Case}
By \eqref{eq:4}, $\{h_i\}_{i=1}^{N-l+1}$ is linearly independent in the module $\cal R\otimes \mathbb{C}^N$,
hence
\begin{equation}\label{eq:8}
\begin{aligned}
\left (\left|C^{(1)}\right|,\cdots,\left|C^{(k)}\right|,\cdots,\left|C^{(N)}\right|\right)
&=\left(\sum\limits_{i=1}^{N-l+1}r_{i1}h_{i1},\cdots,\sum\limits_{i=1}^{N-l+1}r_{i1}h_{ik},\cdots,\sum\limits_{i=1}^{N-l+1}r_{i1}h_{iN}\right)\\
&=\sum\limits_{i=1}^{N-l+1}r_{i1}\left(h_{i1},\cdots,h_{ik},\cdots,h_{iN}\right)\\
&=\sum\limits_{i=1}^{N-l+1}r_{i1}h_i \neq 0,
\end{aligned}
\end{equation}
where $\left|C^{(i)}\right|$ is the determinant of $C^{(i)}$. By the definition of $C^{(k)},$
  \begin{eqnarray}\label{eq:9}
\left| C^{(v_1)}\right|=\cdots= \left|C^{(v_{N-l})}\right|=0.
\end{eqnarray}
Therefore $\left(\left|C^{(m_1)}\right|,\cdots,\left|C^{(m_{l})}\right|\right)\neq 0.$
 By \eqref{eq:8} and \eqref{eq:9},
$$\sum\limits_{i=1}^{l} \left|C^{(m_i)}\right|\otimes e_{m_i}=\sum\limits_{i=1}^N\left|C^{(i)}\right|\otimes e_i=\sum\limits_{i=1}^{N-l+1}r_{i1}h_i\in \mathcal{M}.$$
 This ensures that
 $$
 \sum\limits_{i=1}^{l}\left|C^{(m_i)}\right|\cdot{P_{\mathcal{M}^\bot}(1\otimes e_{m_i})}={P_{\mathcal{M}^\bot}
 \left(\sum\limits_{i=1}^{l}\left|C^{(m_i)}\right|\otimes e_{m_i}\right)}=0.
 $$
The claim is proved.

$(2)\Rightarrow(1)$: Suppose that $\chi_{\cal R}(\mathcal{M}^\perp)=k$. Then by $(1)\Rightarrow(2)$ there is an open subset $U_1\subseteq \mathbb{B}_d$, which is dense in $\mathbb{B}_d$, such that 
$$
  \dim \{E_\lambda(\mathcal{M}\cap ({\cal R}\otimes\mathbb{C}^N))\}=N-k,~\forall \lambda\in U_1.
$$
Notice that $U_1\cap U\neq \emptyset$, hence we have $k=l$. Thus $\chi(\mathcal{M}^\perp)=l$.
\end{proof}

From Theorem \ref{thm 1.2}, we have the following proposition.
\begin{prop}\label{3 equal}
Let $\mathcal{M}$ be a submodule of $H_{\mathbf{K}}\otimes \mathbb{C}^N$. Then the following statements are equivalent.
 \begin{itemize}
   \item [(1)] $K(\mathcal{M}^\perp)=\chi_\mathcal{R}(\mathcal{M}^\perp)$,
   \item [(2)] there exists an open set $U$ which is dense in $\mathbb{B}_d$ such that for all $\lambda\in U$,
 $$
 E_{\lambda}\mathcal{M}=E_{\lambda}(\mathcal{M}\cap ({\cal R}\otimes \mathbb{C}^N)),
 $$
   \item [(3)]  
   there exists $\lambda_0\in \operatorname{mp}(\mathcal{M})$ and $\{f_i\}_{i=1}^{m}\subset \mathcal{M}\cap(\mathcal{R}\otimes\mathbb{C}^N)$ such that
 $$
 E_{\lambda_0}\mathcal{M}= \operatorname{span}\{f_i(\lambda_0)|1\leq i\leq m\}.
 $$
 \end{itemize}
\end{prop}
\begin{proof}
$(2)\Rightarrow (3)$ is obvious. By Theorem \ref{thm 1.2}, it is easy to see that $(1)\Leftrightarrow(2)$. $(3)\Rightarrow(2)$: Set $n=\operatorname{fd}(\mathcal{M})$. Then from $(3)$, there exist $\{f_i\}_{i=1}^{m}\subset \mathcal{M}\cap({\cal R}\otimes\mathbb{C}^N)$ such that
 $$
 \dim \operatorname{span}\{f_i(\lambda_0)|1\leq i\leq m\}=\operatorname{fd}(\mathcal{M})=n.
 $$
 Obviously, $n\leq m$. Without loss of generality, assume that $\dim\operatorname{span}\{f_i(\lambda_0)|1\leq i\leq n\}=n$. Write $f_i=(f_{i,1},\cdots,f_{i,N})$. Then 
 \begin{eqnarray}
   \operatorname{rank} \left[
                 \begin{array}{cccc}
                   f_{1,1}(\lambda_{0}) & f_{1,2}(\lambda_{0}) & \cdots & f_{1,N}(\lambda_{0}) \\
                   f_{2,1}(\lambda_{0}) &f_{2,2}(\lambda_{0}) & \cdots &f_{2,N}(\lambda_{0}) \\
                   \vdots & \vdots & \ddots & \vdots \\
                    f_{n,1}(\lambda_{0}) & f_{n,2}(\lambda_{0}) & \cdots & f_{n,N}(\lambda_{0})\\
                 \end{array}
               \right]=n.
\end{eqnarray}
   Hence there is $1\leq i_1< \cdots < i_n\leq N$ such that
\begin{eqnarray}
      \operatorname{rank} \left[
                 \begin{array}{cccc}
                   f_{1,i_1}(\lambda_{0}) & f_{1,i_2}(\lambda_{0}) & \cdots & f_{1,i_n}(\lambda_{0}) \\
                   f_{2,i_1}(\lambda_{0}) &f_{2,i_2}(\lambda_{0}) & \cdots &f_{2,i_n}(\lambda_{0}) \\
                   \vdots & \vdots & \ddots & \vdots \\
                    f_{n,i_1}(\lambda_{0}) & f_{n,i_2}(\lambda_{0}) & \cdots & f_{n,i_n}(\lambda_{0})\\
                 \end{array}
               \right]=n.
\end{eqnarray}
   Set
   \begin{eqnarray}\label{det phi}
 f(z)=\left|
                 \begin{array}{cccc}
                   f_{1,i_1}(z) & f_{1,i_2}(z) & \cdots & f_{1,i_n}(z) \\
                   f_{2,i_1}(z) &f_{2,i_2}(z) & \cdots &f_{2,i_n}(z) \\
                   \vdots & \vdots & \ddots & \vdots \\
                    f_{n,i_1}(z) & f_{n,i_2}(z) & \cdots & f_{n,i_n}(z)\\
                 \end{array}
               \right|.
\end{eqnarray}
 Then $f\in \operatorname{Hol}(\mathbb{B}_d)$ and $f(\lambda_0)\neq 0$. Let $U=\mathbb{B}_d\setminus Z(f)$, we have $U$ is an open set which is dense in $\mathbb{B}_d$, and 
\begin{eqnarray}       
       \operatorname{rank} \left[
                 \begin{array}{cccc}
                   f_{1,i_1}(\lambda) & f_{1,i_2}(\lambda) & \cdots & f_{1,i_n}(\lambda) \\
                   f_{2,i_1}(\lambda) &f_{2,i_2}(\lambda) & \cdots &f_{2,i_n}(\lambda) \\
                   \vdots & \vdots & \ddots & \vdots \\
                    f_{n,i_1}(\lambda) & f_{n,i_2}(\lambda) & \cdots & f_{n,i_n}(\lambda)\\
                 \end{array}
               \right]=n, ~\forall \lambda\in U.
\end{eqnarray}
 Hence for all $\lambda\in U$, $\dim\operatorname{span}\{f_i(\lambda)|1\leq i\leq n\}=n$. Since 
 $$
 \operatorname{span}\{f_i(\lambda)|1\leq i\leq n\}\subset E_\lambda (\mathcal{M}\cap(\mathcal{R}\otimes\mathbb{C}^N))\subset E_\lambda \mathcal{M}
 $$ 
 and $\dim E_\lambda \mathcal{M}\leq \operatorname{fd}(\mathcal{M})=n$, we have
 $$
 E_\lambda (\mathcal{M}\cap(\mathcal{R}\otimes\mathbb{C}^N))=E_\lambda \mathcal{M}
 $$
 for $\lambda\in U$, which implies that $(2)$ holds.
 
\end{proof}

Now we prove the first part of Theorem \ref{thm1.4}. For convenience, we restate it as the theorem below.
\begin{thm}
Let $\cal R$ be a unital subalgebra of $\operatorname{Mult}(H_{\bf K})$. If $\cal R$ contains a nonzero ideal of $\operatorname{Mult}(H_{\mathbf{K}})$, then Arveson's GBC formula always holds true unconditionally, i.e. for any nonzero submodule $\cal M$ in $H_{\mathbf{K}}\otimes\mathbb{C}^N$,
$$
K({\cal M}^\perp)=\chi_{\cal R}(\cal M^{\perp})
$$
\end{thm}
\begin{proof}
By \cite[Theorem 0.7]{trent}, there is a Hilbert space $E$ and an inner multiplier operator $\Phi\in B(H_{\mathbf{K}}\otimes E,H_{\mathbf{K}}\otimes\mathbb{C}^N)$, such that $\mathcal{M}=\Phi(H_{\mathbf{K}}\otimes E)$ and $P_\mathcal{M}=\Phi\Phi^*.$ Then for all \( \lambda \in \mathbb{B}_d \),
\[
\operatorname{ran} \Phi(\lambda) = \{ \Phi(\lambda)y : y \in E \} = \{ \Phi(\lambda)f(\lambda) : f \in H_{\mathbf{K}}\otimes E \} = \{ E_\lambda(\Phi f) : f \in H_{\mathbf{K}}\otimes E \} = E_\lambda \mathcal{M}.
\]
Choose an orthonormal basis \(\{ e_i \}\) of \(E\), and put
\[
\varphi_i = \Phi(1\otimes e_i) \in \mathcal{M}\cap(\operatorname{Mult}(H_\mathbf{K}) \otimes \mathbb{C}^N).
\]
Then for every \(\lambda \in \mathbb{B}_d\),
\begin{equation}\label{elm=ranphi}
E_\lambda \mathcal{M} = \operatorname{ran} \Phi(\lambda) = \overline{\operatorname{span}\{\varphi_i(\lambda) : i \geq 1\}}.
\end{equation}
Suppose that $\cal I$ is a nonzero ideal in $\operatorname{Mult}(H_{\mathbf K})$ and $\cal I\subset \cal R$. For a fixed nonzero $g \in \cal I$, by \eqref{elm=ranphi} and $\dim E_\lambda \mathcal{M}\leq N$,
we have for $\lambda\in\mathbb{B}_d\setminus Z(g)$,
$$
E_\lambda \mathcal{M}=\operatorname{span}\{g\varphi_i(\lambda) : i \geq 1\}.
$$
Hence for $\lambda\in\mathbb{B}_d\setminus Z(g)$,
$$
E_\lambda \mathcal{M}\subset E_\lambda (\mathcal{M}\cap (\mathcal{I}\otimes \mathbb{C}^N))\subset E_\lambda (\mathcal{M}\cap (\mathcal{R}\otimes \mathbb{C}^N)).
$$
Then by Proposition \ref{3 equal}, 
$$
\chi_{\mathcal{R}}(\mathcal{M}^\perp)=K(\mathcal{M}^\perp).
$$
\end{proof}

At the end of this section, we will prove the second part of Theorem \ref{thm1.4}. For convenience, we rewrite it as the following theorem.

\begin{thm}\label{AH chis}
   Let $\mathfrak{R}$ be a unital subalgebra of $A(H_{\mathbf{K}})$. Then for every integer $0\leq m\leq N$, there is a submodule $\mathcal{M}\subset H_{\mathbf{K}}\otimes\mathbb{C}^N$ such that 
  $$
  \chi_{\mathfrak{R}}(\mathcal{M}^\perp)-K(\mathcal{M}^\perp)=m.
  $$
\end{thm}
In order to prove Theorem \ref{AH chis}, we need the following lemma, which comes from \cite{interpolating sequence}.
A sequence $\{\lambda_i\}_{i\geq 1}\subset \mathbb{B}_d$ is called an interpolating sequence for $\operatorname{Mult}(H_{\mathbf{K}})$ if, whenever $\{w_i\}_{i\geq 1}$ is a bounded sequence of complex numbers, there is a multiplier $\phi$ such that $\phi(\lambda_i) = w_i$, $i\geq 1$.
\begin{lem}\label{interpolat}
 \cite[Theorem~1.1~and~Lemma~2.2]{interpolating sequence} A sequence $\{\lambda_i\}_{i\geq 1}$ is an interpolation sequence for $\operatorname{Mult}(H_{\mathbf{K}})$ if and only if
    \begin{itemize}
    \item [(1)] The Gram matrix $G$ associated with the sequence is bounded, where
    $$
    G({\lambda_i,\lambda_j})=\frac{\langle \mathbf{K}_{\lambda_i}, \mathbf{K}_{\lambda_j}\rangle}{\|\mathbf{K}_{\lambda_i}\|\|\mathbf{K}_{\lambda_j}\|}.
    $$
    \item [(2)] There is a constant $c>0$, such that
    $
    \sqrt{1-|G({\lambda_i,\lambda_j})|^2}\geq c, ~\forall i\neq j.
    $
  \end{itemize}
\end{lem}

$Proof~of~Theorem~\ref{AH chis}$.
 Set $$g(z)=1-\sum\limits_{n=1}^{\infty}b_nz^{n}, \quad z\in \mathbb D.$$
 By the regularity assumption \eqref{sum bn} on the reproducing kernel $\mathbf{K}$,  $\sum\limits_{n=1}^{\infty}b_n=1$. A straightforward calculation shows that for fixed $0<t<1$, 
 $$
   \lim\limits_{x\rightarrow 1^-}\frac{g(t^2)g(x^2)}{g^2(xt)}=0.
 $$
For $x,t\in(0,1),$ write
 $$
  f_t(x)=\frac{g(t^2)g(x^2)}{g^2(xt)}.
 $$
Then a simple analysis yields a nonnegative sequence $\{c_i\}_{i\ge 1}$ that increases monotonically to $1$ and satisfies
$$
f_{c_i}(c_j)\leq \frac{1}{(j-1)j^2},~\forall 1\leq i < j.
$$
  Let $\{\zeta_i\}_{i\geq 1}$ be a countable dense subset of $\partial\mathbb B_d$, $G$ be the Gram matrix associated to $\{c_i\zeta_i\}_{i=1}^\infty$. Obviously 
  $$
  G(c_i\zeta_i,c_i\zeta_i)=1, \,\forall i\geq 1.
  $$ Moreover,  for $1\leq i<j$,
 \begin{equation}\label{Gij<f}
 |G(c_i\zeta_i,c_j\zeta_j)|^2
 =\frac{g(|c_i\zeta_i|^{2})g(|c_j\zeta_j|^{2})}{\left|g(\langle c_i\zeta_i,c_j\zeta_j\rangle)\right|^2}
 \leq f_{c_i}(c_j)
 \leq \frac{1}{(j-1)j^2}.
 \end{equation}
 Thus,
 the Gram matrix $G$ associated with the sequence $\{c_i\zeta_i\}_{i\geq 1}$ satisfies
\begin{equation}\label{GH-S}
  \|G-I\|_{HS}^2
\leq 2\sum_{j=1}^{\infty}\sum_{i=1}^{j-1}|G(c_i\zeta_i,c_j\zeta_j)|^2\\
  \leq 2\sum_{j=1}^{\infty}\frac{1}{j^2}\\
  =\frac{\pi^2}{3},
\end{equation}
which implies that $G$ is bounded. Moreover, by \eqref{Gij<f} again, 
\begin{equation}\label{dHlambda}
\sqrt{1-|G({c_i\zeta_i,c_j\zeta_j})|^2}\geq {\frac{\sqrt{3}}{2}}, ~\forall 1\leq i<j.
\end{equation}
By Lemma \ref{interpolat}, $\{c_i\zeta_i\}_{i\geq 1}$ is an interpolation sequence of $\operatorname{Mult}(H_{\mathbf{K}})$, and hence  there is an $f\in \operatorname{Mult}(H_{\mathbf{K}})$ such that
\begin{equation}\label{f neq 0 in M}
  f(c_1\zeta_1)=1~\text{and}~f(c_i\zeta_i)=0~\text{for}~i\geq 2.
\end{equation}
Set 
$$
\mathbf{M}=\{f\in H_{\mathbf{K}}:f(c_i\zeta_i)=0,~\forall i\geq 2\}.
$$ 
Notice that $\overline{\{c_i\zeta_i\}_{i=1}^\infty}\cap\partial\mathbb{B}_d=\partial\mathbb B_d$ and $\mathfrak{R}\subseteq A(H_{\mathbf{K}}) \subseteq \operatorname{Mult}(H_{\mathbf{K}})\cap C(\overline{\mathbb{B}}_d)$. Therefore
\begin{equation}\label{MR}
 \mathbf{M}\cap \mathfrak{R}=\{0\}.
\end{equation}
Let
\[
\mathcal{M} = \underbrace{\mathbf{M} \oplus \cdots \oplus \mathbf{M}}_{m\text{ times}} \oplus \underbrace{H_{\mathbf{K}} \oplus \cdots \oplus H_{\mathbf{K}}}_{N-m\text{ times}}.
\]
By \eqref{MR}, it is easy to see that 
$$
\operatorname{fd}(\mathcal{M}\cap (\mathfrak{R}\otimes\mathbb{C}^N))\leq N-m.
$$
If $m=N$, then $\operatorname{fd}(\mathcal{M}\cap (\mathfrak{R}\otimes\mathbb{C}^N))=0$. If $m<N$, since $\mathfrak{R}\neq 0$, there is a $g\in \mathfrak{R}$ and a $\lambda_0\in \mathbb{B}_d$ such that $g(\lambda_0)\neq 0$. Then
$$
\dim\{E_{\lambda_0}(\mathcal{M}\cap (\mathfrak{R}\otimes\mathbb{C}^N))\} \geq \dim \{g(\lambda_0)e_{m+1},\cdots, g(\lambda_0)e_{N}\}=N-m.
$$
Hence $\operatorname{fd}(\mathcal{M}\cap (\mathfrak{R}\otimes\mathbb{C}^N))=N-m$. By Theorem \ref{thm 1.2}, 
\begin{equation}\label{MR1}
 \chi_{\mathfrak{R}}(\mathcal{M}^\perp)=N-\operatorname{fd}(\mathcal{M}\cap (\mathfrak{R}\otimes\mathbb{C}^N))=m.
\end{equation}
From \eqref{f neq 0 in M},  it follows that $\mathbf{M} \neq 0,$ hence $\operatorname{fd}(\mathcal{M})=N$. By Lemma \ref{K=N-fd}, $K({\cal M}^\perp)=0$. It follows that 
$$
\chi_{\mathfrak{R}}(\mathcal{M}^\perp)-K(\mathcal{M}^\perp)=m.
$$
\hfill$\qedsymbol$

\section{The finite defect problem: finitely-many-variable case}

In this section, we will study the finite defect problem for submodules in $H_{\mathbf{K}}\otimes\mathbb C^N$, and  we will prove Theorem \ref{thm1.5}. Recall that a submodule $\cal M$ in $H_{\mathbf{K}}\otimes\mathbb C^N$ is called of finite defect, if the defect operator $\Delta_{\cal M}$ is of finite rank. Let $\mathbf{K}(\lambda, z)$ be a regular unitarily invariant CNP reproducing kernel on the unit ball $\mathbb{B}_d$ $(d\ge 2)$ satisfying
$$
1-{\frac{1}{\mathbf{K}_{\lambda}(z)}}=\sum\limits_{n=1}^\infty b_n\langle z,\lambda\rangle^n,
$$ 
where $b_n\geq 0$.
 For $\eta\in \mathbb C^N$,  set 
\begin{equation}\label{Ceta}
 \mathbb{C}_\eta=\{\lambda\eta|\lambda\in \mathbb{C}\}.
\end{equation}
  Moreover, we need the following representation of the defect operator:
 \begin{eqnarray}\label{defect operator}
 \Delta_{\mathcal{M}}^{2}=P_{\mathcal{M}}\left(E_0+\sum_{\alpha\in \mathbb{N}_+^d}b_{\alpha}[M_{z^\alpha}\otimes I,P_{\mathcal{M}^{\perp}}][P_{\mathcal{M}^{\perp}},M_{z^\alpha}^{*}\otimes I]\right)P_{\mathcal{M}}
 \end{eqnarray}
In fact, notice that
$$
\begin{aligned}
&P_\mathcal{M}[M_{z^\alpha}\otimes I,P_{\mathcal{M}^{\perp}}][P_{\mathcal{M}^{\perp}},M_{z^\alpha}^{*}\otimes I]P_\mathcal{M}\\
=&P_\mathcal{M}\left((M_{z^\alpha}\otimes I)P_{\mathcal{M}^{\perp}}-P_{\mathcal{M}^{\perp}}(M_{z^\alpha}\otimes I)\right)
\left(P_{\mathcal{M}^{\perp}}(M_{z^\alpha}^{*}\otimes I)-(M_{z^\alpha}^{*}\otimes I)P_{\mathcal{M}^{\perp}}\right)P_\mathcal{M}\\
=&P_\mathcal{M}(M_{z^\alpha}\otimes I)P_{\mathcal{M}^{\perp}}(M_{z^\alpha}^{*}\otimes I)P_\mathcal{M}\\
=&P_\mathcal{M}(M_{z^\alpha}\otimes I)(M_{z^\alpha}^{*}\otimes I)P_\mathcal{M}-P_\mathcal{M}(M_{z^\alpha}\otimes I)P_{\mathcal{M}}(M_{z^\alpha}^{*}\otimes I)P_\mathcal{M}.
\end{aligned}
$$
Hence by (\ref{defect frac}), 
\begin{equation*}
\begin{aligned}
&\Delta_{\mathcal{M}}^{2}\\
=&P_\mathcal{M}-\sum_{\alpha\in \mathbb{N}_+^d}b_{\alpha}P_\mathcal{M}(M_{z^\alpha}\otimes I)P_\mathcal{M}(M_{z^\alpha}^{*}\otimes I)P_\mathcal{M}\\
=&P_\mathcal{M}-\sum_{\alpha\in \mathbb{N}_+^d}b_{\alpha}P_\mathcal{M}(M_{z^\alpha}\otimes I)(M_{z^\alpha}^{*}\otimes I)P_\mathcal{M}+\sum_{\alpha\in \mathbb{N}_+^d}b_{\alpha}P_\mathcal{M}[M_{z^\alpha}\otimes I,P_{\mathcal{M}^{\perp}}][P_{\mathcal{M}^{\perp}},M_{z^\alpha}^{*}\otimes I]P_\mathcal{M}\\
=&P_{\mathcal{M}}\left(I-\sum_{\alpha\in \mathbb{N}_+^d}b_{\alpha}(M_{z^\alpha}\otimes I)(M_{z^\alpha}^{*}\otimes I)+\sum_{\alpha\in \mathbb{N}_+^d}b_{\alpha}[M_{z^\alpha}\otimes I,P_{\mathcal{M}^{\perp}}][P_{\mathcal{M}^{\perp}},M_{z^\alpha}^{*}\otimes I]\right)P_{\mathcal{M}}\\
=&P_{\mathcal{M}}\left(E_0+\sum_{\alpha\in \mathbb{N}_+^d}b_{\alpha}[M_{z^\alpha}\otimes I,P_{\mathcal{M}^{\perp}}][P_{\mathcal{M}^{\perp}},M_{z^\alpha}^{*}\otimes I]\right)P_{\mathcal{M}}.
\end{aligned}
\end{equation*}

We will divide Theorem \ref{thm1.5} in two cases, depending on the set $\{n: b_n>0\}$ being finite or infinite. For  the set $\{n; b_n>0\}$ being finite, we have 

\begin{thm}\label{finite codim}
Suppose $\{n: b_n>0\}$ is a finite set. Then for any submodule $\mathcal{M}$ in $H_{\mathbf{K}}\otimes \mathbb C^N$,  $\operatorname{rank}\mathcal{M}<\infty$ if and only if there is a subspace $F\subset \mathbb{C}^N$ such that  
$$\mathcal{M}\subset H_{\mathbf{K}}\otimes F \, \text{and }\dim H_{\mathbf{K}}\otimes F/\mathcal{M}<\infty.$$
\end{thm}
    To prove Theorem \ref{finite codim}, we need the following lemma, which may be known before, and we write here for convenience.
\begin{lem}\label{finite codim equal polynomial}
Let $F$ be a finite-dimensional Hilbert space. Then for a submodule $\mathcal{M} $ in $ H_{\mathbf{K}}\otimes F$,  $$\dim H_{\mathbf{K}}\otimes F/\mathcal{M}<\infty$$ if and only if for any $1\leq i\leq d$ and $\eta\in F\setminus \{0\}$,
   $\mathcal{M}\cap \left(\mathbb C[z_i]\otimes \mathbb{C}_\eta\right)\not=0$.
\end{lem}
\begin{proof} Let $$H_{z_i}=\{f\in H_{\mathbf{K}}:~ f\text{ is a function on the variable }z_i\}.$$
For necessity, suppose $\dim H_{\mathbf{K}}\otimes F/\mathcal{M}<\infty$, then for each $1\leq i\leq d$ and $\eta\in F$,  $$\dim H_{z_i}\otimes \mathbb{C}_\eta/(\mathcal{M}\cap \left(H_{z_i}\otimes \mathbb{C}_\eta\right))<\infty.$$ By \cite[Theorem 2.2.5]{CG}, ${\cal M}\cap \left(H_{z_i}\otimes \mathbb C_\eta\right)$ contains a nonzero polynomial. 
  For sufficiency, suppose  for each $1\leq i\leq d$ and $1\leq j\leq \dim F$,  the space $\mathcal{M}\cap \left(H_{z_i}\otimes \mathbb{C}_{e_j}\right)$ contains a nonzero polynomial $p_{i,j}(z_i)\otimes e_j$. Let ${\cal M}_0$ be the submodule generated by $\{p_{i,j}(z_i)\otimes e_j; i=1,\cdots,d,\, j=1,\cdots, \dim F\}$. Obviously, $\mathcal{M}_0 $ is a finitely codimensional submodule in $H_{\mathbf{K}}$ and $\mathcal{M}_0\subset\cal M$, and hence $\cal M$ is of finite codimension. 
\end{proof}

$Proof~of~Theorem~\ref{finite codim}$. Sufficiency. Write $\mathcal{N}=H_{\mathbf{K}}\otimes F\ominus\mathcal{M}$, then $\dim \mathcal{N}<\infty$. By the assumption,  $\{n: b_n>0\}$ is a finite set. Combining with the representation (\ref{defect operator}),
there is an $N_0$ such that 
 $$
 \Delta_{\mathcal{M}}^{2}=P_{\mathcal{M}}\left(E_0+\sum_{|\alpha|\leq N_0}b_{\alpha}[M_{z^\alpha}\otimes I,P_{\mathcal{N}}][P_{\mathcal{N}},M_{z^\alpha}^{*}\otimes I]\right)P_{\mathcal{M}}.
 $$
It follows from  $\dim \mathcal{N}<\infty$ that for $\alpha\in\mathbb N^d_+$, $[M_{z^\alpha}\otimes I,P_{\mathcal{N}}]$ are of finite rank, and hence
$$
\dim \operatorname{ran}\Delta_\mathcal{M}^2\leq \dim F+\sum_{|\alpha|\leq N_0}\dim \operatorname{ran}[M_{z^\alpha}\otimes I,P_{\mathcal{N}}][P_{\mathcal{N}},M_{z^\alpha}^{*}\otimes I]<\infty,
$$
which implies that $\operatorname{rank}\Delta_\mathcal{M}<\infty$.

For necessity, assume  $\operatorname{rank} \Delta_\mathcal{M}<\infty$.  For $\eta\in \mathbb{C}^N\setminus\{0\}$, let $Q_\eta=I\otimes P_{\mathbb C_\eta}$, where $P_{\mathbb C_\eta}$ is the orthogonal projection from $\mathbb C^N$ onto $\mathbb C_\eta$. Let $E=\{\eta\in \mathbb{C}^N: Q_\eta \mathcal{M}=0\}$ and $F=\mathbb{C}^N\ominus E$, then it is easy to see that $\mathcal{M}$ is of finite defect as a submodule of $H_{\mathbf{K}}\otimes F$. Without loss of generality, assume that $F=\mathbb{C}^N$.

\noindent\emph{Claim.} For all $ 1\leq i\leq d, \eta\in \mathbb{C}^N\setminus\{0\}$, $\mathcal{M}\cap \left(H_{z_i}\otimes \mathbb{C}_\eta\right)$ contains a nonzero polynomial.

First, the claim will be proved for $d=2$, and then we will use the induction argument to finish the general case. 

\noindent\emph{Step 1. Proof of the claim for $d=2$}. 

Decompose $\mathcal{M}\ominus z_1\mathcal{M}$ and $\mathcal{M}\ominus z_2\mathcal{M}$ as
\begin{equation}\label{M-z1M}
\mathcal{M}\ominus z_1\mathcal{M}=((\mathcal{M}\ominus z_1\mathcal{M})\cap (H_{z_2}\otimes \mathbb{C}^N))\oplus \Lambda_1,
\mathcal{M}\ominus z_2\mathcal{M}=((\mathcal{M}\ominus z_2\mathcal{M})\cap (H_{z_1}\otimes \mathbb{C}^N))\oplus \Lambda_2.
\end{equation}
Similar to the proof of \cite[Theorem 3.1]{Gu}, we have
\begin{equation}\label{dim lambda 1}
\dim \Lambda_i<\infty,~i=1,2.
\end{equation}
Next we prove that for all $\eta\in \mathbb{C}^N\setminus \{0\}$, $Q_\eta ((\mathcal{M}\ominus z_1\mathcal{M})\cap (H_{z_2}\otimes \mathbb{C}^N))\neq 0$. Suppose that there exists $\eta\in \mathbb{C}^N\setminus \{0\}$ such that $Q_\eta ((\mathcal{M}\ominus z_1\mathcal{M})\cap (H_{z_2}\otimes \mathbb{C}^N))= 0$, then by \eqref{M-z1M}, 
$$
Q_\eta\mathcal{M}\subset \overline{z_1Q_\eta\mathcal{M}}+Q_\eta\Lambda_1. 
$$
Write $\mathcal{M}_\eta=\overline{Q_\eta\mathcal{M}}$, then
$$
\mathcal{M}_\eta\subset \overline{z_1\mathcal{M}_\eta}+Q_\eta\Lambda_1,
$$
which implies that 
\begin{equation}\label{M-z1Mdim<inf}
\dim \mathcal{M}_\eta/\overline{z_1\mathcal{M}_\eta}<\infty. 
\end{equation}
Since $\mathcal{M}_\eta$ is a submodule of $H_{\mathbf{K}}\otimes\mathbb{C}_\eta$, by \cite[Lemma 3.3]{Gu}, $\dim \mathcal{M}_\eta\ominus z_i\mathcal{M}_\eta=\infty$ for each $1\leq i\leq 2$, which contradicts \eqref{M-z1Mdim<inf}. Hence 
\begin{equation}\label{3.1.star}
  Q_\eta( (\mathcal{M}\ominus z_1\mathcal{M})\cap (H_{z_2}\otimes \mathbb{C}^N))\neq 0,~\forall\eta\in \mathbb{C}^N\setminus \{0\}
\end{equation}
For $f\in (\mathcal{M}\ominus z_1\mathcal{M})\cap (H_{z_2}\otimes \mathbb{C}^N)$, write
$$
f=\sum_{n=0}^{\infty}z_2^n\otimes v_n^{(f)}.
$$
Let
$$
Q=\{v_n^{(f)}:~f\in (\mathcal{M}\ominus z_1\mathcal{M})\cap (H_{z_2}\otimes \mathbb{C}^N),~n\geq 0\}.
$$
Then it is easy to see that $Q=\mathbb{C}^N$. Otherwise, there exists $\eta\in \mathbb{C}^N\setminus\{0\}$ such that
$$
z_2^n\otimes\eta\perp(\mathcal{M}\ominus z_1\mathcal{M})\cap (H_{z_2}\otimes \mathbb{C}^N),~\forall n\geq 0.
$$
Hence $H_{z_2}\otimes\mathbb{C}_\eta\perp (\mathcal{M}\ominus z_1\mathcal{M})\cap (H_{z_2}\otimes \mathbb{C}^N)$, which implies that
$$
Q_\eta ((\mathcal{M}\ominus z_1\mathcal{M})\cap (H_{z_2}\otimes \mathbb{C}^N))=0,
$$
contradicting \eqref{3.1.star}. Therefore there exist $f_1,\cdots,f_N\subset (\mathcal{M}\ominus z_1\mathcal{M})\cap (H_{z_2}\otimes \mathbb{C}^N)$ and $n_1,\cdots,n_N\geq 0$ such that 
\begin{equation}\label{span fj=CN}
 \operatorname{span}\{ v_{n_j}^{(f_j)}:~1\leq j\leq N\}=\mathbb{C}^N.
\end{equation}
Now we show that for fixed $1\leq j\leq N$, there is an $h_j\in \cal M$ such that $h_j(z_1,0)$ is a nonzero polynomial in $\mathbb C[z_1]\otimes \mathbb{C}_{v_{n_j}^{(f_j)}}$. We prove it by considering two cases.\\
\textbf{Case 1. } $n_j=0$.\par
We can set $h_j(z_1,z_2)=f_j(z_2)$, then $h_j(z_1,0)=f_j(0)=v_{n_j}^{(f_j)}\neq 0$.\\
\textbf{Case 2. } $n_j>0$.\par
By (\ref{1.a intro}), it is easy to see that
\begin{equation}\label{b1=a1>0}
  b_1=a_1>0.
\end{equation}
Then by (\ref{defect operator}), $\dim\operatorname{span}\{P_{\mathcal{M}^\perp}(M_{z_2}^*\otimes I) z_1^{k}f_j|k\geq 1\}<\infty$.
Set
$$
m_0=\dim\operatorname{span}\{P_{\mathcal{M}^\perp}(M_{z_2}^*\otimes I) z_1^{k}f_j:~k\geq 1\}+1.
$$
Thus for each $l\in\mathbb{N}^+$, there are $\lambda_1^{(l,1)},\lambda_2^{(l,1)},\cdots,\lambda_{m_0}^{(l,1)}$, which are not all zero, such that
\begin{equation}\label{4.a}
  P_{\mathcal{M}^\perp}(M_{z_2}^*\otimes I)\left(\sum_{i=1}^{m_0}\lambda_i^{(l,1)}z_1^{m_0(l-1)+i}f_j\right)=\sum_{i=1}^{m_0}\lambda_i^{(l,1)}P_{\mathcal{M}^\perp}(M_{z_2}^*\otimes I)z_1^{m_0(l-1)+i}f_j=0.
\end{equation}
Write 
\begin{equation}\label{4.c}
  g_l^{(1)}=\sum\limits_{i=1}^{m_0}\lambda_i^{(l,1)}z_1^{m_0(l-1)+i}.
\end{equation}
Since $f_j\in \cal M$, we have $g_l^{(1)}f_j\in \mathcal{M}$. By (\ref{4.a}),
$$
(M_{z_2}^*\otimes I)g_l^{(1)}f_j\in \mathcal{M}.
$$ 
Noticing  $\operatorname{rank}\Delta_{\mathcal{M}}<\infty$ and $b_1=a_1\neq 0$, by (\ref{defect operator}) again, 
$
 \operatorname{span}\{P_{\mathcal{M}^\perp}(M_{z_2}^*\otimes I)^2g_k^{(1)}f_j|k\geq 1\}
$ is of finite dimension.
Set
$$
m_1=\dim \operatorname{span}\{P_{\mathcal{M}^\perp}(M_{z_2}^*\otimes I)^2g_k^{(1)}f_j:~k\geq 1\}+1.
$$
Thus for each $l\in\mathbb{N}_+$, there are $\lambda_1^{(l,2)},\lambda_2^{(l,2)},\cdots,\lambda_{m_1}^{(l,2)}$, which are not all zero, such that
\begin{equation}\label{PM g_t^2}
  P_{\mathcal{M}^\perp}(M_{z_2}^*\otimes I)^2\sum_{i=1}^{m_1}\lambda_i^{(l,2)}g_{m_1(l-1)+i}^{(1)}f_j=0.
\end{equation}
Write 
$$
g_l^{(2)}=\sum\limits_{i=1}^{m_1}\lambda_i^{(l,2)}g_{m_1(l-1)+i}^{(1)}.
$$
By (\ref{4.c}), it is easy to verify that $\{g_{m_1(l-1)+i}^{(1)}\}_{i\geq 1}$ are mutually orthogonal. Hence $g_l^{(2)}\neq 0$.
Then by (\ref{PM g_t^2}), we have
$$
g_l^{(2)}f_j\in \mathcal{M}~\text{and}~(M_{z_2}^*\otimes I)^2g_l^{(2)}f_j\in \mathcal{M}.
$$
By repeating the above operation $n_j$ times, we obtain 
$g_l^{(n_j)}\in \mathbb{C}[z_1]$ satisfying
\begin{equation}\label{ga}
  g_l^{(n_j)}f_j\in \mathcal{M},~g_l^{(n_j)}\neq 0~\text{and}~(M_{z_2}^*\otimes I)^{n_j}g_l^{(n_j)}f_j\in \mathcal{M}.
\end{equation}
Since $(M_{z_2}^*)^{n_j}g_l^{(n_j)}z_2^{k}=0$ for $k<n_j$ and $((M_{z_2}^*)^{n_j}g_l^{(n_j)}z_2^{k})(z_1,0)=0$ for $k>n_j$, we have
\begin{equation}\label{3.1.12}
((M_{z_2}^*\otimes I)^{n_j}g_l^{(n_j)}f_j)(z_1,0)=((M_{z_2}^*)^{n_j}g_l^{(n_j)}z_2^{n_j})\otimes v_{n_j}^{(f_j)}\in \mathbb{C}[z_1]\otimes \mathbb{C}_{v_{n_j}^{(f_j)}}.
\end{equation}
Then from $\left\langle (M_{z_2}^*)^{n_j}g_l^{(n_j)}z_2^{n_j}, g_l^{(n_j)}\right\rangle =\|z_2^{n_j}g_l^{(n_j)}\|^2>0$ and $(M_{z_2}^*)^{n_j}g_l^{(n_j)}z_2^{n_j}\in \mathbb{C}[z_1]$, we have
\begin{equation}\label{eq:3.13}
\left((M_{z_2}^*)^{n_j}g_l^{(n_j)}z_2^{n_j}\right)(z_1,0)\neq 0.
\end{equation}
Set $h_j=(M_{z_2}^*\otimes I)^{n_j}g_l^{(n_j)}f_j$. Then by \eqref{3.1.12} and \eqref{eq:3.13},
$$
h_j(z_1,0)\in \mathbb{C}[z_1]\otimes\mathbb{C}_{v_{n_j}^{(f_j)}}~\text{and}~h_j(z_1,0)\neq 0.
$$

Let 
$$
G_j=\{ph_j:~p\in \mathbb{C}[z_1]\}.
$$
By (\ref{dim lambda 1}), $P_{\Lambda_2}|_{G_j}:G_j\rightarrow \Lambda_2$ is of finite rank. So there exists a nonzero polynomial $\tilde{p}_j\in\mathbb{C}[z_1]$ such that
\begin{equation}\label{P2 ph=0}
  P_{\Lambda_2}(\tilde{p}_jh_j)=0.
\end{equation}
Thus by (\ref{M-z1M}) and (\ref{P2 ph=0}),
$$
\tilde{p}_jh_j\in\left((\mathcal{M}\ominus z_2\mathcal{M})\cap \left(H_{z_1}\otimes\mathbb{C}_{v_{n_j}^{(f_j)}}\right)\right)\oplus \overline{z_2\mathcal{M}}=\left(\mathcal{M}\cap \left(H_{z_1}\otimes\mathbb{C}_{v_{n_j}^{(f_j)}}\right)\right)\oplus \overline{z_2\mathcal{M}},
$$
therefore there exist $\varphi_j\otimes v_{n_j}^{(f_j)}\in \mathcal{M}\cap \left(H_{z_1}\otimes\mathbb{C}_{v_{n_j}^{(f_j)}}\right)$ and $\zeta_j\in \overline{z_2\mathcal{M}}$ such that
$$
\tilde{p}_j(z_1)h_j(z_1,z_2)=\varphi_j(z_1)\otimes v_{n_j}^{(f_j)}+\zeta_j(z_1,z_2).
$$
Let $z_2=0$, then
$
\varphi_j(z_1)\otimes v_{n_j}^{(f_j)}=\tilde{p}_j(z_1)h_j(z_1,0)\neq 0,
$
which implies that 
$$
\varphi_j(z_1)\otimes v_{n_j}^{(f_j)}\in \left(\mathcal{M}\cap \left(\mathbb{C}[z_1]\otimes\mathbb{C}_{v_{n_j}^{(f_j)}}\right)\right)\setminus\{0\}.
$$
By \eqref{span fj=CN}, for $\eta\in \mathbb{C}^N\setminus\{0\}$, write $\eta=\sum\limits_{j=1}^{N}\eta_jv_{n_j}^{(f_j)}$, then
$$
\left(\prod_{j=1}^{N}\varphi_j\right)\otimes\eta
=\sum_{j=1}^{N}\left(\eta_j\prod_{\substack{k=1\\k\neq j}}^{N}\varphi_k\right)\cdot\left(\varphi_j\otimes v_{n_j}^{(f_j)}\right)\in \mathcal{M}\cap (\mathbb{C}[z_1]\otimes \mathbb{C}_\eta).
$$
Since $\left(\prod\limits_{j=1}^{N}\varphi_j\right)\neq 0$, we have for all $\eta\in \mathbb{C}^N\setminus\{0\}$, $\mathcal{M}\cap (\mathbb{C}[z_1]\otimes \mathbb{C}_\eta)\neq \{0\}$. 
By the same reasoning, for all $\eta\in \mathbb{C}^N\setminus\{0\}$, $\mathcal{M}\cap (\mathbb{C}[z_2]\otimes \mathbb{C}_\eta)\neq \{0\}$. Consequently, the claim holds when $d=2$. 

\noindent\emph{Step 2. Proof of the claim for $d\geq 3$}.

Now using induction, assume the claim holds for the $d-1$ dimensional case. Let 
$$
H_{z_{i_1},\cdots,z_{i_k}}=\{f\in H_{\mathbf{K}}:~f\text{ is a function on the variable }z_{i_1},\cdots,z_{i_k}\}.
$$
Decompose $\mathcal{M}\ominus z_d\mathcal{M}$ as
\begin{equation}\label{M-zdM}
\mathcal{M}\ominus z_d\mathcal{M}=((\mathcal{M}\ominus z_d\mathcal{M})\cap (H_{z_1,\cdots,z_{d-1}}\otimes\mathbb{C}^N))\oplus \Lambda_d
\end{equation}
For $\eta\neq 0$, since $\dim\mathcal{M}/z_d\mathcal{M}\geq \dim \overline{Q_\eta\mathcal{M}}/\overline{z_dQ_\eta\mathcal{M}}$,
similar to the proof of \cite[Theorem 3.1]{Gu}, we have
\begin{equation}\label{dim lambda d}
\dim \Lambda_d<\infty
\end{equation}
By \cite[Lemma 3.3]{Gu}, $\dim \mathcal{M}\ominus z_d\mathcal{M}=\infty$, which implies that
$$
(\mathcal{M}\ominus z_d\mathcal{M})\cap (H_{z_1,\cdots,z_{d-1}}\otimes\mathbb{C}^N)\neq 0.
$$
Then $\mathbf{M}_{d}:=(\mathcal{M}\ominus z_d\mathcal{M})\cap (H_{z_1,\cdots,z_{d-1}}\otimes\mathbb{C}^N)$ can be regarded as a nonzero submodule of $H_{z_1,\cdots,z_{d-1}}\otimes\mathbb{C}^N$. Notice that
$$
[P_{\mathbf{M}_{d}},M_{z^\alpha}]=P_{\mathbf{M}_{d}}M_{z^\alpha}P_{\mathbf{M}_{d}^\perp}
=P_{\mathbf{M}_{d}}P_{\mathcal{M}}M_{z^\alpha}P_{\mathcal{M}^\perp+\Lambda_d}P_{\mathbf{M}_{d}^\perp}.
$$
Hence 
$$
\operatorname{rank}b_\alpha[P_{\mathbf{M}_{d}},M_{z^\alpha}]\leq \operatorname{rank}b_\alpha[P_{\mathcal{M}},M_{z^\alpha}]+\dim\Lambda_d.
$$
Then from $\operatorname{rank}\Delta_{\mathcal{M}}<\infty$ and $\{n: b_n>0\}$ is a finite set, by \eqref{defect operator}, $\operatorname{rank}\Delta_{\mathbf{M}_{d}}<\infty$. By the induction hypothesis applied to \(\mathbf M_d\),  there are nonzero polynomials $p_{i,j}(z_i)\otimes e_j \in \mathbf{M}_{d}\subset\mathcal{M}$ for each $1\leq i\leq d-1$ and $1\leq j\leq N$. Let $\mathbf{M}_{1}:=(\mathcal{M}\ominus z_1\mathcal{M})\cap (H_{z_2,\cdots,z_{d}}\otimes\mathbb{C}^N)$. Using the same reasoning, we get nonzero polynomials $p_{d,j}(z_d)\otimes e_j \in \mathbf{M}_{1}\subset\mathcal{M}$ for $1\leq j\leq N$. Then for all $1\leq i\leq d$ and $\eta=\sum\limits_{j=1}^{N}\eta_je_j\in \mathbb{C}^N$,
$$
\left(\prod_{j=1}^{N}p_{i,j}\right)\otimes\eta
=\sum_{j=1}^{N}\left(\eta_j\prod_{\substack{k=1\\k\neq j}}^{N}p_{i,k}\right)\cdot(p_{i,j}\otimes e_j)\in \mathcal{M}\cap (\mathbb{C}[z_i]\otimes \mathbb{C}_\eta).
$$
For each $1\leq i\leq d$,
since $\left(\prod\limits_{j=1}^{N}p_{i,j}\right)\neq 0$, we have for all $\eta\in \mathbb{C}^N\setminus\{0\}$, $\mathcal{M}\cap (\mathbb{C}[z_i]\otimes \mathbb{C}_\eta)\neq \{0\}$.
By induction the claim holds for $d<\infty$.

Hence by Lemma \ref{finite codim equal polynomial}, $\dim H_{\mathbf{K}}\otimes\mathbb{C}^N/\mathcal{M}<\infty$.

\hfill \qedsymbol
\begin{rem}
  The argument extends without difficulty to the case where $H_{\mathbf{K}}\otimes \mathbb{C}^N$ is replaced by $H_{\mathbf{K}}\otimes H$ for any (infinite-dimensional) Hilbert space $H$.
\end{rem}

To continue, we need the following lemma.
\begin{lem}\label{Delte Meta le inf}
  Let $\mathcal{M}$ be a submodule of $H_{\mathbf{K}}\otimes \mathbb{C}^N$. If $\operatorname{rank}\Delta_{\mathcal{M}}<\infty$, then for all $\eta\in \mathbb{C}^N\setminus\{0\}$, $\operatorname{rank}\Delta_{\mathcal{M}_{\eta}}<\infty$.
\end{lem}
\begin{proof}
By \eqref{defect operator}, we have
\begin{equation}\label{3.3.star4}
\langle \Delta_{\mathcal{M}}^2 f, f \rangle 
= \| E_0 f \|^2 + \sum\limits_{\alpha\in\mathbb{N}_+^{d}} \|\sqrt{b_\alpha} P_{\mathcal{M}^\perp} (M_{z^\alpha}\otimes I)^* f \|^2,~f\in\mathcal{M},
\end{equation}
and
\begin{equation}\label{3.3.star5}
\langle \Delta_{\mathcal{M}_\eta}^2 f, f \rangle 
= \| E_0 f \|^2 + \sum\limits_{\alpha\in\mathbb{N}_+^{d}} \|\sqrt{b_\alpha} P_{\mathcal{M}_\eta^\perp} (M_{z^\alpha}\otimes I)^* f \|^2,~f\in\mathcal{M}_\eta.
\end{equation}
For \(f \in \ker \Delta_{\mathcal{M}}\), by \eqref{3.3.star4}, $E_0f=0$ and $\sqrt{b_\alpha}(M_{z^\alpha}\otimes I)^* f\in \mathcal{M}$. Hence
$$
E_0 (Q_\eta f) = Q_\eta (E_0 f) = 0
$$
and
$$
\sqrt{b_\alpha}(M_{z^\alpha}\otimes I)^* (Q_\eta f) = Q_\eta (\sqrt{b_\alpha}(M_{z^\alpha}\otimes I)^* f) \in Q_\eta \mathcal{M} \subset \mathcal{M}_\eta,~\forall \alpha\in\mathbb{N}_+^{d},
$$
Then by $\eqref{3.3.star5}$, \(Q_\eta f \in \ker \Delta_{\mathcal{M}_\eta}\). Therefore 
\begin{equation}\label{3.3.1}
Q_\eta \ker \Delta_{\mathcal{M}} \subset \ker \Delta_{\mathcal{M}_\eta}.
\end{equation}
Since $\operatorname{rank}\Delta_{\mathcal{M}}<\infty$, $\dim \mathcal{M}/\ker\Delta_{\mathcal{M}}<\infty$. Thus by \eqref{3.3.1},
$$
\dim \mathcal{M}_\eta/\ker \Delta_{\mathcal{M}_\eta}
\leq \dim \overline{Q_\eta\mathcal{M}}/\overline{Q_\eta\ker \Delta_{\mathcal{M}}}
\leq \dim (Q_\eta\mathcal{M})/(Q_\eta\ker\Delta_{\mathcal{M}})
\leq \dim \mathcal{M}/\ker\Delta_{\mathcal{M}}<\infty,
$$ 
which implies that $\operatorname{rank}\Delta_{\mathcal{M}_{\eta}}<\infty$.
\end{proof}
\begin{lem}\label{no polynomial}
Suppose that $\{n: b_n>0\}$ is an infinite set. Let $F=\mathbb{C}^N\ominus\{\zeta\in \mathbb{C}^N:~\mathcal{M}_\zeta= 0\}$. If $\cal M$ is a submodule in $H_{\mathbf{K}}\otimes \mathbb{C}^N$ of finite defect, then 
\begin{itemize}
\item[1.] for $\forall \eta \in F\setminus\{0\}$, $P_{\cal M} (1\otimes\eta)\not =0$;
\item[2.] $(H_{\mathbf{K}}\otimes F)\ominus \mathcal{M}$ has no nonzero polynomial.
\end{itemize}
\end{lem}
\begin{proof}
First, for  $\forall \eta\in F\setminus\{0\}$, we will prove that $P_{\cal M} (1\otimes\eta)\not=0$ by using contradiction argument.  Assume that there is $\eta\in F\setminus\{0\}$ such that $P_{\cal M} (1\otimes\eta)=0$, i.e. $1\otimes \eta\in \mathcal{M}^\perp\cap (H_{\mathbf{K}}\otimes\mathbb C_\eta)$.  Since for $f\in H_{\mathbf{K}}\otimes \mathbb{C}_\eta$,
$$
\left\langle f, Q_\eta g\right\rangle=\left\langle f,g\right\rangle,~\forall g\in \mathcal{M},
$$  
we have 
\begin{eqnarray}\label{eq_Neta}
(H_{\mathbf{K}}\otimes \mathbb C_{\eta})\ominus \mathcal{M}_\eta=\mathcal{M}^\perp\cap \left(H_{\mathbf{K}}\otimes \mathbb{C}_\eta\right).
\end{eqnarray}
Hence $1\otimes\eta\perp \mathcal{M}_\eta$, which implies that $\forall h=h_0\otimes \eta\in\mathcal{M}_\eta$,
\begin{eqnarray}\label{4.1}
h(0)=E_0h=\langle h_0\otimes {\eta},1\otimes{\eta/\|\eta\|}\rangle 1\otimes\eta/\|\eta\|=0, 
\end{eqnarray}
where $E_n: H_{\mathbf{K}}\otimes \mathbb{C}^N\to H_{n}$ be the orthogonal projection with $H_n$ being defined as 
 $$H_{n}=\{f\in H_{\mathbf{K}}\otimes \mathbb{C}^N;  f\text{ is homogeneous and~} deg f=n\}.$$ Moreover, for any $g\in H_{\mathbf{K}}\otimes \mathbb{C}^N\setminus\{0\}$, set 
\begin{eqnarray}\label{eq_Lg}
L_{g}=\inf \{n\in \mathbb N;\, E_n g\not=0\}\,\text{and }L_{\mathcal{M}_\eta}=\inf\{L_g;\, g\in {\mathcal{M}_\eta}\setminus\{0\}\}, 
\end{eqnarray}
then $L_{\mathcal{M}_\eta}\geq 1$ and there is an $f_0\in\mathcal{M}_\eta$, s.t. $L_{f_0}=L_{\mathcal{M}_\eta}$. Choose $1\leq k\leq d$ such that 
\begin{eqnarray}\label{f_0}
(M_{z_k}^*\otimes I)E_{L_{\mathcal{M}_\eta}} f_0\not=0,
\end{eqnarray}
which implies that 
\begin{equation}\label{3.4.star6}
 L_{(M_{z_k}^*\otimes I) f_0}=L_{(M_{z_k}^*\otimes I)E_{L_{\mathcal{M}_\eta}} f_0}
\end{equation}
and
\begin{eqnarray}\label{4.11}
L_{(M_{z_k}^*\otimes I) f_0}=L_{\mathcal{M}_\eta}-1\geq 0.
\end{eqnarray}
Write ${\cal N}_\eta=(H_{\mathbf{K}}\otimes \mathbb{C}_\eta)\ominus\mathcal{M}_\eta$. Recall that 
$$
\Delta_{\mathcal{M}_\eta}^2=P_{\mathcal{M}_\eta}\left(E_0 +\sum\limits_{\alpha\in\mathbb N_+^d}b_{\alpha} [M_{z^\alpha}\otimes I, P_{{\cal N}_\eta}][ P_{{\cal N}_\eta}, M_{z^\alpha}^*\otimes I]\right)P_{\mathcal{M}_\eta}.
$$
By (\ref{4.1}) for any $h\in \mathcal{M}_\eta$, $$P_{\mathcal{M}_\eta}E_0 P_{\mathcal{M}_\eta}h=P_{\mathcal{M}_\eta}E_0 h=P_{\mathcal{M}_\eta} h(0)=0.$$
 Thus, $h\in \ker \Delta_{\mathcal{M}_\eta}=\ker\Delta_{\mathcal{M}_\eta}^2$ if and only if for any $\alpha\in \mathbb N_+^d$, 
$$
b_{\alpha} [ P_{{\cal N}_\eta}, M_{z^\alpha}^*\otimes I]h=0.
$$
By the assumption, $S=\{n: b_n\not=0\}$ is an infinite set, and we can write $S=\{n_i\}_{i=1}^\infty$ with $1\leq n_1<n_2<\cdots$. By Lemma \ref{Delte Meta le inf}, $\operatorname{rank} {\mathcal{M}_\eta}<\infty$, then for any fixed $1\leq j\leq d$, the subspace 
$$
\operatorname{span}\{ P_{{\cal N}_\eta} (M_{z_k}^*\otimes I) z_j^{n_i}f_0: \, i=1,2,\cdots\}
$$
is of finite dimension. Denote 
$$
m=\dim \operatorname{span}\{ P_{{\cal N}_\eta} (M_{z_k}^*\otimes I) z_j^{n_i} f_0: \, i=1,2\cdots\}+1.
$$
It follows that for any $t\in\mathbb N$, $\{ P_{{\cal N}_\eta} (M_{z_k}^*\otimes I) z_j^{n_i} f_0:\,i=1,\cdots, m\}$ is linearly dependent, that is, there are $\{\lambda_i^{(t)}\}_{i=1}^{m}$ which are not all zero, such that
$$
 P_{{\cal N}_\eta} (M_{z_k}^*\otimes I)\left(\sum\limits_{i=1}^m\lambda_i^{(t)}z_{j}^{n_{mt+i}}f_0\right)=\sum\limits_{i=1}^m \lambda_i^{(t)}\left( P_{{\cal N}_\eta}(M_{z_k}^*\otimes I) z_{j}^{n_{mt+i}}f_0\right)=0.
$$
To simplify the notation, set
$$
g_t=\sum\limits_{i=1}^m\lambda_i^{(t)}z_{j}^{n_{mt+i}}f_0,
$$
then both $g_t$ and $(M_{z_k}^*\otimes I) g_t$ are in $\mathcal{M}_\eta$. Since
$\{\Delta_{\mathcal{M}_\eta}M_{z_k}^{*}g_t\}_{t=1}^{\infty}$ is linearly dependent, there is a positive integer $l$ and complex numbers $\mu_{1},\cdots,\mu_{l}$, which are not all zero, such that
\begin{equation}\label{3.4.star3}
\sum_{t=1}^{l}\mu_{t}\Delta_{\mathcal{M}_\eta}(M_{z_k}^*\otimes I)g_t=0.
\end{equation}
By \eqref{3.3.star5} and \eqref{3.4.star3}, it is easy to see that
\begin{equation}\label{1.2}
b_{\alpha}P_{{\cal N}_\eta}(M_{z^{\alpha}}^{*}\otimes I)\sum_{t=1}^{l}\mu_{t}(M_{z_k}^*\otimes I)g_t=0,~\text{for all}~\alpha\neq 0.
\end{equation}
It follows from $b_{n_i}\not=0$ that 
$$
(M_{z_{j}^{n_i}}^*\otimes I)\sum_{t=1}^{l}\mu_{t}(M_{z_k}^*\otimes I)g_t\in\mathcal{M}_\eta.
$$
Set 
$$
t_0=\min\{t: \mu_t\not=0, 1\leq t\leq l\},\, 
\text{and } i_{t_0}=\min\{i: \lambda_{i}^{(t_0)}\not=0, 1\leq i\leq m\}.
$$
Then $(M_{z_j^{n_{mt_0+i_{t_0}}}}^*\otimes I)\sum\limits_{t=1}^{l}\mu_{t}(M_{z_k}^*\otimes I)g_t\in \mathcal{M}_\eta$. By the definition  (\ref{eq_Lg}) of $L_{\mathcal{M}_\eta}$,
\begin{equation}\label{E_r-1=0}
E_{L_{\mathcal{M}_\eta}-1}\left((M_{z_j^{n_{mt_0+i_{t_0}}}}^*\otimes I)\sum_{t=1}^{l}\mu_{t}(M_{z_k}^*\otimes I)g_t\right)=0.
\end{equation}
Since $E_{i}(M_{z_j}^*\otimes I)=(M_{z_j}^*\otimes I)E_{i+1}$, we have
\begin{equation}\label{Er-1=Erf}
\begin{aligned}
&E_{L_{\mathcal{M}_\eta}-1}\left((M_{z_j^{n_{mt_0+i_{t_0}}}}^*\otimes I)\sum_{t=1}^{l}\mu_{t}(M_{z_k}^*\otimes I)g_t\right)\\
=&(M_{z_j^{n_{mt_0+i_{t_0}}}}^*\otimes I)\sum_{t=1}^{n}\mu_{t}(M_{z_k}^*\otimes I)\sum_{i=1}^{m}\lambda_{i}^{(t)} (M_{z_j^{n_{mt+i}}}\otimes I)E_{r+n_{mt_0+i_{t_0}}-n_{mt+i}}f_0\\
=&(M_{z_j^{n_{mt_0+i_{t_0}}}}^*\otimes I)(M_{z_k}^*\otimes I)\sum_{t=1}^{l}\sum_{i=1}^{m}\mu_{t}\lambda_{i}^{(t)}(M_{z_j^{n_{mt+i}}}\otimes I)E_{L_{\mathcal{M}_\eta}+n_{mt_0+i_{t_0}}-n_{mt+i}}f_0\\
=&(M_{z_j^{n_{mt_0+i_{t_0}}}}^*\otimes I)(M_{z_k}^*\otimes I)\sum_{mt+i<mt_0+i_{t_0},t\leq l,i\leq m}\mu_{t}\lambda_{i}^{(t)}(M_{z_j^{n_{mt+i}}}\otimes I)E_{L_{\mathcal{M}_\eta}+n_{mt_0+i_{t_0}}-n_{mt+i}}f_0\\
&+\mu_{t_0}\lambda_{i_{t_0}}^{(t_0)}(M_{z_j^{n_{mt_0+i_{t_0}}}}^*\otimes I)(M_{z_k}^*\otimes I)(M_{z_j^{n_{mt_0+i_{t_0}}}}\otimes I)E_{L_{\cal M}}f_0\\
&+(M_{z_j^{n_{mt_0+i_{t_0}}}}^*\otimes I)(M_{z_k}^*\otimes I)\sum_{mt+i>mt_0+i_{t_0},t\leq l,i\leq m}\mu_{t}\lambda_{i}^{(t)}(M_{z_j^{n_{mt+i}}}\otimes I)E_{L_{\mathcal{M}_\eta}+n_{mt_0+i_{t_0}}-n_{mt+i}}f_0\\
\overset{(*)}=&0+\mu_{t_0}\lambda_{i_{t_0}}^{(t_0)}(M_{z_j^{n_{mt_0+i_{t_0}}}}^*\otimes I)(M_{z_k}^*\otimes I)(M_{z_j^{n_{mt_0+i_{t_0}}}}\otimes I)E_{L_{\mathcal{M}_\eta}}f_0+0\\
=&\mu_{t_0}\lambda_{i_{t_0}}^{(t_0)}(M_{z_j^{n_{mt_0+i_{t_0}}}}^*\otimes I)(M_{z_k}^*\otimes I)(M_{z_j^{n_{mt_0+i_{t_0}}}}\otimes I)E_{L_{\mathcal{M}_\eta}}f_0,
\end{aligned}
\end{equation}
where $(*)$ follows from the facts that 
$$
\mu_{t}\lambda_{i}^{(t)}=0, ~mt+i<mt_0+i_{t_0},
$$
and 
$$
E_{L_{\mathcal{M}_\eta}+n_{mt_0+i_{t_0}}-n_{mt+i}}f_0=0, ~mt+i>mt_0+i_{t_0}.
$$
By (\ref{E_r-1=0}) and (\ref{Er-1=Erf}),
\begin{equation}\label{1.3}
(M_{z_j^{n_{mt_0+i_{t_0}}}}^*\otimes I)(M_{z_k}^*\otimes I)(M_{z_j^{n_{mt_0+i_{t_0}}}}\otimes I)E_{L_{\mathcal{M}_\eta}}f_0=0.
\end{equation}
Write
\begin{equation}\label{Erf}
E_{L_{\mathcal{M}_\eta}}f_0=f_{1}+f_{2},
\end{equation}
where $f_{1}\in \operatorname{ker}(M_{z_k}^*\otimes I)\cap (H_{L_{\mathcal{M}_\eta}}\otimes\mathbb{C}_\eta)$, $f_{2}\in\overline{\operatorname{ran}(M_{z_k}\otimes I)}\cap (H_{L_{\mathcal{M}_\eta}}\otimes\mathbb{C}_\eta)$. Notice that $$(M_{z_k}^*\otimes I)(M_{z_j^{n_{mt_0+i_{t_0}}}}^*\otimes I)(M_{z_j^{n_{mt_0+i_{t_0}}}}\otimes I)f_{1}=0.$$ 
Then by (\ref{1.3}) and (\ref{Erf}),
\begin{equation}\label{f2}
(M_{z_k}^*\otimes I)(M_{z_j^{n_{mt_0+i_{t_0}}}}^*\otimes I)(M_{z_j^{n_{mt_0+i_{t_0}}}}\otimes I)f_2=0.
\end{equation}
Since $f_2\in\overline{\operatorname{ran}(M_{z_k}\otimes I)}\cap (H_{L_{\mathcal{M}_\eta}}\otimes\mathbb{C}_\eta)$, by (\ref{f2}), we have
$$
\left\langle (M_{z_j^{n_{m{t_0}+i_{t_0}}}}^*\otimes I)(M_{z_j^{n_{mt_0+i_{t_0}}}}\otimes I)f_2, f_2\right\rangle=0,
$$
which implies that $f_2=0$. Therefore $E_{L_{\mathcal{M}_\eta}}f_0=f_1\in \ker (M_{z_k}^*\otimes I)$,
which contradicts \eqref{f_0}.

Now, we will prove $(H_{\mathbf{K}}\otimes F)\ominus\mathcal{M}$ has no nonzero polynomial. By using contradiction argument again, suppose that there exists a nonzero polynomial $p\in (H_{\mathbf{K}}\otimes F)\ominus\mathcal{M}$. Write
  $$
  p=\sum_{\alpha\in \mathbb{N}^d}z^{\alpha}\otimes c_{\alpha},~c_{\alpha}\in F.
  $$
Choose an $\alpha_0\in \mathbb{N}^d$ such that $|\alpha_0|=deg(p)$ and $c_{\alpha_0}\neq 0$. Obviously, $z^{\alpha_0}P_{\mathcal{M}}(1\otimes c_{\alpha_0})\in\cal M$. Then
$$
\begin{aligned}
0=&\left\langle z^{\alpha_0}P_{\mathcal{M}}(1\otimes c_{\alpha_0}) , p \right\rangle\\
=&\left\langle P_{\mathcal{M}}(1\otimes c_{\alpha_0}), (M_{z^{\alpha_0}}^{*}\otimes I)p \right\rangle\\
=&\left\langle P_{\mathcal{M}}(1\otimes c_{\alpha_0}), (M_{z^{\alpha_0}}^{*}\otimes I) (z^{\alpha_0}\otimes c_{\alpha_0})\right\rangle =\left\langle P_{\mathcal{M}}(1\otimes c_{\alpha_0}), \|z^{\alpha_0}\|^2\otimes c_{\alpha_0}  \right\rangle\\
=&\frac{ {\alpha_0}!}{a_{|{\alpha_0}|}|{\alpha_0}|!}  \left\|P_{\mathcal{M}}(1\otimes c_{\alpha_0}) \right\|^{2}>0,
\end{aligned}
$$
which is a contradiction.
\end{proof}

\begin{thm}\label{M=Hk}
Suppose that $\{n: b_n>0\}$ is an infinite set. Let $\cal M$ be a nonzero submodule in $H_{\mathbf{K}}\otimes \mathbb{C}^N$. Then $\mathcal{M}$ is of finite defect if and only if there is a subspace $F\subset \mathbb{C}^N$ such that ${\cal M}=H_{\mathbf{K}}\otimes F$.
 \end{thm}
\begin{proof}
The sufficiency is obvious, and we now prove the necessity. Write 
$$
\{n:\, b_n>0\}=\{n_i\}_{i=1}^\infty, \qquad 1\le n_1<n_2<\cdots,
$$ 
and $F=\mathbb{C}^N\ominus\{\zeta\in \mathbb{C}^N|\mathcal{M}_\zeta= 0\}$, then $\mathcal{M}\subset H_{\mathbf{K}}\otimes F$. Since $\operatorname{rank}\cal M<\infty$, for any fixed $\eta\in F\setminus\{0\}$, the space 
  $$
  \operatorname{span}\{\Delta_{\mathcal{M}}^{2}P_{\mathcal{M}}(z_1^{n_i}\otimes\eta)|i=1,2,\cdots\}
  $$ 
  is of finite dimension. Let
$$
m=\dim\operatorname{span}\{\Delta_{\mathcal{M}}^{2}P_{\mathcal{M}}(z_1^{n_i}\otimes\eta)|i=1,2,\cdots\}+1.
$$
Hence there exist $\lambda_1,\cdots,\lambda_m \in\mathbb C$, which are not all zero such that
\begin{equation}\label{3.5.star}
\Delta_{\mathcal{M}}^2 P_{\mathcal{M}}\sum_{i=1}^{m}\lambda_{i}z_1^{n_i}\otimes\eta=\sum_{i=1}^{m}\lambda_{i}\Delta_{\mathcal{M}}^{2}P_{\mathcal{M}}(z_1^{n_i}\otimes\eta)=0.
\end{equation}
Let $m_0=\sup\{i:\lambda_i\neq 0\}$ and 
$$
q(z_1)=\sum_{i=1}^{m_0}\lambda_{i}z_1^{n_i}\otimes\eta.
$$
By \eqref{3.3.star4} and \eqref{3.5.star}, for $\forall \alpha\not=0$, 
\begin{equation}\label{1.4}
b_{|\alpha|}P_{\mathcal{M}^\perp}(M_{z^{\alpha}}\otimes I)^*P_{\mathcal{M}} q=0,
\end{equation}
that is, for any $\beta\in\mathbb N^d$ and $\zeta\in \mathbb{C}^N$,
$$
\left\langle P_{\mathcal{M}^\perp}b_{|\alpha|}(M_{z^{\alpha}}\otimes I)^*P_{\mathcal{M}}q, z^{\beta}\otimes\zeta\right\rangle=0,~\forall \alpha\neq 0.
$$
We claim that $P_{\mathcal{M}} q$ is a polynomial. Notice that
\begin{equation}\label{z beta-PM z beta}
\begin{aligned}
0&=\left\langle P_{\mathcal{M}^\perp}b_{|\alpha|}(M_{z^{\alpha}}\otimes I)^*P_{\mathcal{M}} q, z^{\beta}\otimes\zeta\right\rangle\\
&=b_{|\alpha|}\left(\left\langle (M_{z^{\alpha}}\otimes I)^*P_{\mathcal{M}} q, z^{\beta}\otimes\zeta\right\rangle-\left\langle (M_{z^{\alpha}}\otimes I)^*P_{\mathcal{M}} q, P_{\mathcal{M}} (z^{\beta}\otimes\zeta)\right\rangle\right),
\end{aligned}
\end{equation}
and for $|\alpha|>n_{m_0}$, 
\begin{equation}\label{alpha > nm0}
\begin{aligned}
\left\langle (M_{z^{\alpha}}\otimes I)^*P_{\mathcal{M}} q, P_{\mathcal{M}} (z^{\beta}\otimes\zeta)\right\rangle
&=\left\langle (M_{z^{\alpha}}\otimes I)^*\left(I-P_{\mathcal{M}^\perp}\right)q, P_{\mathcal{M}} (z^{\beta}\otimes\zeta)\right\rangle\\
&=\left\langle (M_{z^{\alpha}}\otimes I)^* q, P_{\mathcal{M}} (z^{\beta}\otimes\zeta)\right\rangle
=0.
\end{aligned}
\end{equation}
Hence by (\ref{z beta-PM z beta}) and (\ref{alpha > nm0}), for any $\alpha, \beta\in\mathbb N^d $ with $|\alpha|> n_{m_0}$ and any $\zeta\in\mathbb{C}^N$,
\begin{equation*}
b_{|\alpha|}\left\langle P_{\mathcal{M}} q, z^{\alpha}z^{\beta}\otimes\zeta\right\rangle=0.
\end{equation*}
In particular, since $b_{n_{m_0+1}}\not=0$, for any $\alpha, \beta\in\mathbb N^d$ with $|\alpha|=n_{m_0+1}$ and $\zeta\in\mathbb{C}^N$,
\begin{eqnarray}\label{z alpha z beta = 0}
\left\langle P_{\mathcal{M}} q, z^{\alpha}z^{\beta}\otimes\zeta\right\rangle=0.
\end{eqnarray}
Please note that for any $\gamma\in\mathbb N^d$ with $|\gamma|>n_{{m_0}+1}$, there exist $\gamma_1,\gamma_2\in\mathbb N^d$ with $|\gamma_1|=n_{{m_0}+1}$ such that $z^{\gamma}=z^{\gamma_1}z^{\gamma_2}$. It follows from (\ref{z alpha z beta = 0}) that, for $\forall \gamma\in\mathbb N^d$, $\zeta\in\mathbb{C}^N$ and $|\gamma|>n_{m_0+1}$,
$$
\left\langle P_{\mathcal{M}}q, z^{\gamma}\otimes\zeta\right\rangle=0,
$$
which implies that $P_{\mathcal{M}} q$ is a polynomial and the claim is proved. Notice that $P_{\mathcal{M}^\perp}q=q-P_{\mathcal{M}}  q$, which is also a polynomial. 
By Lemma $\ref{no polynomial}$, $P_{\mathcal{M}^\perp} q=0$, and hence 
\begin{equation}\label{1.a}
q=P_{\mathcal{M}}q.
\end{equation}
Then by $(\ref{1.4})$ and $(\ref{1.a})$, we have
$$
\begin{aligned}
P_{\mathcal{M}^\perp}(1\otimes\eta)&= P_{\mathcal{M}^\perp} a_{n_{m_0}} (M_{z_{1}^{n_{m_0}}}^*\otimes I)(z_1^{n_{m_0}}\otimes\eta)\\
&=\frac{a_{n_{m_0}}}{{\lambda_{m_0}}} P_{\mathcal{M}^\perp}(M_{z_{1}^{n_{m_0}}}^*\otimes I)q\\
&=\frac{a_{n_{m_0}}}{{\lambda_{m_0}}}  P_{\mathcal{M}^\perp}(M_{z_{1}^{n_{m_0}}}^*\otimes I)P_{\mathcal{M}}q=0,
\end{aligned}
$$
i.e. $1\otimes\eta \in\mathcal{M}$. By the arbitrariness of $\eta$, $\mathcal{M}\supset H_{\mathbf{K}}\otimes F$, hence $\mathcal{M}= H_{\mathbf{K}}\otimes F$.

\end{proof}
\section{Rigidity problem for submodules of finite codimension}
The problem of rigidity for unitary equivalence of submodules originated from the classical Beurling theorem, which implies that all the submodules in $H^2(\mathbb D)$ are unitarily equivalent. For the Hardy module over the polydisc, Agrawal, Clark and Douglas \cite{ACD} initiated the systematic study; Douglas and Yan \cite{DY}, Douglas, Paulsen, Sah and Yan \cite{DPSY} developed an algebraic reduction approach, and  Guo's Characteristic Space Theory \cite{Guo1, Guo2} provided a local analytic-algebraic framework. However, for Bergman modules and Dirichlet modules over the unit balls, Richter \cite{Ric} and Putinar \cite{Pu} proved that ${\cal M}_1$ is unitarily equivalent to ${\cal M}_2$ if and only if ${\cal M}_1={\cal M}_2$. For the general case of reproducing kernel Hilbert modules, this problem is still open. 

Let $\Omega$ be a bounded domain in $\mathbb C^d$ ($d\geq 2)$, and $H_{\mathfrak{K}}$ be a reproducing kernel Hilbert space of holomorphic functions in $\Omega$, with the kernel $\mathfrak{K}:\Omega\times\Omega\to\mathbb C$. Assume that 
\begin{itemize}
  \item [(1)] $\mathcal{P}_d$ is dense in $H_{\mathfrak{K}}$.
  \item [(2)] the coordinate multipliers
$
        M_{z_1},\ldots,M_{z_d}
$
are bounded on \(H_{\mathfrak{K}}\).
  \item [(3)] $
         \mathfrak{K}(\lambda,\lambda)\longrightarrow\infty,
        \text{ as } \lambda\to\partial\Omega .
$
\end{itemize}
In this section, we prove Theorem \ref{thm:rigidity}. For reader's convenience, we rewrite the theorem as follows.
\begin{thm}\label{thm:vector-valued-rigidity}
Let \(\mathcal{M}_1,\mathcal{M}_2\subset H_{\mathfrak{K}}\otimes \mathbb C^N\) be finite-codimensional submodules. Suppose there exists
a unitary module isomorphism
$$
U:\mathcal{M}_1\to \mathcal{M}_2
$$
such that
$$
U\bigl(M_{z_j}\otimes I|_{\mathcal{M}_1}\bigr)
=
\bigl(M_{z_j}\otimes I|_{\mathcal{M}_2}\bigr)U,
\quad j=1,\dots,d.
$$
Then there exists a constant unitary matrix $W\in M_N(\mathbb{C})$ such that ${\cal M}_2=(I_{H_{\mathfrak{K}}}\otimes W)\mathcal{M}_1$. 
\end{thm}
To prove this theorem, we need some preparations. For a submodule $\mathcal{M}\subset H_{\mathfrak{K}}\otimes\mathbb{C}^N$, set
\begin{eqnarray}\label{def:zero_points}
\mathcal{Z}(\mathcal{M})=\{\lambda\in\Omega:\dim\mathcal{M}_\lambda<N\},
\end{eqnarray}
where $\mathcal{M}_\lambda=\{f(\lambda):f\in\mathcal{M}\}$.
Denote by
$$
\mathcal{Q}=\mathcal{P}_d\otimes\mathbb{C}^N/(\mathcal{M}\cap(\mathcal{P}_d\otimes\mathbb{C}^N)), \,
\operatorname{Ann}_{\mathcal{P}_d}(\mathcal{Q})=\{p\in\mathcal{P}_d:p\cdot\mathcal{P}_d\otimes\mathbb{C}^N\subseteq\mathcal{M}\cap(\mathcal{P}_d\otimes\mathbb{C}^N)\},
$$
and
$$
A(\mathcal{M})=\{\lambda\in\Omega:\exists\,p\in\mathcal{P}_d~\text{such that}~ p(\lambda)\neq0,\ p\otimes e_i\in\mathcal{M}\text{ for }1\leq i\leq N\}.
$$

\begin{lem}\label{prop:scheme_vector_free_off_support}
For a polynomially generated submodule $\mathcal{M}\subset H_{\mathfrak{K}}\otimes\mathbb{C}^N$, we have
$$
\mathcal{Z}(\mathcal{M})=\Omega\setminus A(\mathcal{M})=Z(\operatorname{Ann}_{\mathcal{P}_d}(\mathcal{Q}))\cap\Omega,
$$
where $Z(\operatorname{Ann}_{\mathcal{P}_d}(\mathcal{Q}))=\{\lambda\in\mathbb{C}^d|~p(\lambda)=0,~\forall p\in \operatorname{Ann}_{\mathcal{P}_d}(\mathcal{Q})\}$.
\end{lem}

\begin{proof}
Obviously, 
$$
  \mathcal{Z}(\mathcal{M})\subset \Omega\setminus A(\mathcal{M})\subset Z(\operatorname{Ann}_{\mathcal{P}_d}(\mathcal{Q}))\cap\Omega.
$$
Hence it is sufficient to prove that $Z(\operatorname{Ann}_{\mathcal{P}_d}(\mathcal{Q}))\cap\Omega\subset \mathcal{Z}(\mathcal{M})$. Since $\mathcal{M}$ is polynomially generated, 
\begin{equation}\label{eq:fibre}
\mathcal{M}_\lambda=\{p(\lambda):p\in\mathcal{M}\cap(\mathcal{P}_d\otimes\mathbb{C}^N)\},\quad\lambda\in\Omega.
\end{equation}
Let $\mathfrak{m}_\lambda=\{p\in\mathcal{P}_d:p(\lambda)=0\}$. Localising at $\mathfrak{m}_\lambda$,
the evaluation at $\lambda$ induces a homomorphism
$$
\begin{aligned}
\phi_\lambda: (\mathcal{P}_d\otimes\mathbb{C}^N)_{\mathfrak{m}_\lambda}/\mathfrak{m}_\lambda(\mathcal{P}_d\otimes\mathbb{C}^N)_{\mathfrak{m}_\lambda}&\to \mathbb{C}^N\\
\left[\frac{p}{s}\right]&\mapsto \frac{p(\lambda)}{s(\lambda)},
\end{aligned}
$$
where 
$$
(\mathcal{P}_d\otimes\mathbb{C}^N)_{\mathfrak{m}_\lambda}=(\mathcal{P}_d\setminus {\mathfrak{m}_\lambda})^{-1}(\mathcal{P}_d\otimes\mathbb{C}^N).
$$
We claim that $\phi_\lambda$ is isomorphic. Since $\phi_\lambda([\frac{1\otimes e_i}{1}])=e_i$, $\phi_\lambda$ is surjective. Suppose that $\frac{p_1(\lambda)}{s_1(\lambda)}=\frac{p_2(\lambda)}{s_2(\lambda)}$, then $\frac{p_1s_2-p_2s_1}{s_1s_2}(\lambda)=0$, which implies that $\frac{p_1s_2-p_2s_1}{s_1s_2}\in \mathfrak{m}_\lambda(\mathcal{P}_d\otimes\mathbb{C}^N)_{\mathfrak{m}_\lambda}$, hence $\left[\frac{p_1}{s_1}\right]=\left[\frac{p_2}{s_2}\right]$, therefore $\phi_\lambda$ is an injective. The claim is proved. Then the natural embedding $(\mathcal{M}\cap(\mathcal{P}_d\otimes\mathbb{C}^N))_{\mathfrak{m}_\lambda}\hookrightarrow (\mathcal{P}_d\otimes\mathbb{C}^N)_{\mathfrak{m}_\lambda}$ induces a homomorphism
$$
\varphi_\lambda: (\mathcal{M}\cap(\mathcal{P}_d\otimes\mathbb{C}^N))_{\mathfrak{m}_\lambda}/{\mathfrak{m}_\lambda}(\mathcal{M}\cap(\mathcal{P}_d\otimes\mathbb{C}^N))_{\mathfrak{m}_\lambda}\hookrightarrow
(\mathcal{P}_d\otimes\mathbb{C}^N)_{\mathfrak{m}_\lambda}/\mathfrak{m}_\lambda(\mathcal{P}_d\otimes\mathbb{C}^N)_{\mathfrak{m}_\lambda}\cong\mathbb{C}^N.
$$
Since 
$$
\operatorname{Im}\varphi_\lambda=\left\{\frac{p(\lambda)}{s(\lambda)}:~p\in \mathcal{M}\cap(\mathcal{P}_d\otimes\mathbb{C}^N),~s\in \mathcal{P}_d\setminus\mathfrak{m}_\lambda\right\},
$$
by \eqref{eq:fibre}, $\operatorname{Im}\varphi_\lambda=\mathcal{M}_{\lambda}$. Suppose that $\lambda\in \Omega\setminus\mathcal{Z}(\mathcal{M})$, then $\mathcal{M}_{\lambda}=\mathbb{C}^N$. Hence $\varphi_\lambda$ is surjective, which implies that
\begin{equation}\label{4301}
(\mathcal{M}\cap(\mathcal{P}_d\otimes\mathbb{C}^N))_{\mathfrak{m}_\lambda}+\mathfrak{m}_\lambda(\mathcal{P}_d\otimes\mathbb{C}^N)_{\mathfrak{m}_\lambda}=(\mathcal{P}_d\otimes\mathbb{C}^N)_{\mathfrak{m}_\lambda}.
\end{equation}
Then Nakayama's lemma \cite[Corollary 2.7]{Atiyah} yields
$$
(\mathcal{M}\cap(\mathcal{P}_d\otimes\mathbb{C}^N))_{\mathfrak{m}_\lambda}=(\mathcal{P}_d\otimes\mathbb{C}^N)_{\mathfrak{m}_\lambda},
$$
Hence by \cite[Proposition 3.3]{Atiyah}
\begin{equation}\label{4302}
\mathcal{Q}_{\mathfrak{m}_\lambda}=0. 
\end{equation}
Since $\mathcal{Q}$ is a finitely generated $\mathcal{P}_d$-module, by \cite[Exercise 3.19]{Atiyah},
$$
\operatorname{Supp}(\mathcal{Q})=\{\mathfrak{p}\in\operatorname{Spec}\mathcal{P}_d:\mathfrak{p}\supseteq \operatorname{Ann}_{\mathcal{P}_d}(\mathcal{Q})\},
$$
where $\operatorname{Supp}(\mathcal{Q})=\{\mathfrak{p}\in\operatorname{Spec}\mathcal{P}_d:\mathcal{Q}_{\mathfrak{p}}\neq 0\}$ and $\operatorname{Spec}\mathcal{P}_d$ is the prime spectrum of $\mathcal{P}_d$. By \eqref{4302},
$\mathfrak{m}_\lambda\notin\operatorname{Supp}(\mathcal{Q})$, hence $\mathfrak{m}_\lambda\not\supseteq \operatorname{Ann}_{\mathcal{P}_d}(\mathcal{Q})$, which implies that
$
\lambda\notin Z(\operatorname{Ann}_{\mathcal{P}_d}(\mathcal{Q})).
$
Therefore
$$
Z(\operatorname{Ann}_{\mathcal{P}_d}(\mathcal{Q}))\cap\Omega\subset \mathcal{Z}(\mathcal{M}).
$$
\end{proof}
For a submodule $\mathcal{M}$ of $H_{\mathfrak{K}}\otimes\mathbb{C}^N$, set
    \begin{equation}\label{bianhao111}
   \text{mp}(\mathcal{M})=\{\lambda \in \Omega: \operatorname{dim}E_{\lambda}\mathcal{M}=\operatorname{fd}(\mathcal{M})\},
    \end{equation}
A submodule $\mathcal{M}$ is said to be locally algebraic if there exists a $\lambda_{0}\in \operatorname{mp}(\mathcal{M})$ and polynomials $\{p_i\}_{i=1}^m \subseteq \mathcal{M}$ such that
  $$
  E_{\lambda_{0}}\mathcal{M}=\operatorname{span}\{p_i(\lambda_{0}):1\leq i\leq m\}.
  $$
\begin{lem}\label{prop:vector_local_multiplier_unified}
Let $\mathcal{M}_1,\mathcal{M}_2\subset H_{\mathfrak{K}}\otimes \mathbb{C}^N$ be locally algebraic submodules, and $\mathcal{Z}(\mathcal{M}_1)\cup \mathcal{Z}(\mathcal{M}_2)\subsetneq \Omega$. For simplicity, set $\mathbf{Z}=\mathcal{Z}(\mathcal{M}_1)\cup \mathcal{Z}(\mathcal{M}_2)$. Let
\[
U:\mathcal{M}_1\to \mathcal{M}_2
\]
be a continuous module isomorphism, then there exist holomorphic matrix-valued functions
\[
\Phi,\Psi:\Omega\setminus \mathbf{Z} \to M_{N}(\mathbb C),
\]
such that
\[
(Uf)(\lambda)=\Phi(\lambda)f(\lambda),
\quad \lambda\in \Omega\setminus \mathbf{Z},\ f\in \mathcal{M}_1,
\]
and
\[
(U^{-1}g)(\lambda)=\Psi(\lambda)g(\lambda),
\quad \lambda\in \Omega\setminus \mathbf{Z},\ g\in \mathcal{M}_2.
\]
 Moreover, 
\[
\Psi(\lambda)\Phi(\lambda)=I_{\mathbb{C}^N}=\Phi(\lambda)\Psi(\lambda),
\quad \lambda\in \Omega\setminus \mathbf{Z}.
\]
\end{lem}

\begin{proof}
By the definition (\ref{def:zero_points}) of $\mathcal{Z}(\mathcal{M}_1)$,  for $\lambda\not\in \mathcal{Z}(\mathcal{M}_1)$, $\dim \mathcal{M}_1{_\lambda}=N$.
Since $\mathcal{M}_1$ is locally algebraic and $\mathcal{Z}(\mathcal{M}_1)\subsetneq \Omega$, there is $\lambda_0\in \Omega\setminus\mathcal{Z}(\mathcal{M}_1)$ and $\{p_{j}\}_{j=1}^N\subset \mathcal{M}_1\cap\mathcal{P}_d\otimes \mathbb C^N$ such that
$$
\dim \operatorname{span}\{p_{j}(\lambda_0):1\leq j\leq N\}=N.
$$
Let $\mathcal{I}_1=[p_{1},\cdots,p_N]\subset \mathcal{M}_1$ be the submodule of $H_\mathfrak{K}\otimes\mathbb{C}^N$ generated by $\{p_1,\cdots,p_N\}$, then $\lambda_0\in \Omega\setminus \mathcal{Z}(\mathcal{I}_1)$. By Lemma~\ref{prop:scheme_vector_free_off_support} there exists a polynomial $p_{\lambda_0} \in \mathcal{P}_d$ such that $p_{\lambda_0}(\lambda_0) \neq 0$ and $p_{\lambda_0} \otimes e_j \in \mathcal{I}_1$ for all $j = 1,\dots,N$. For $\lambda\in \Omega\setminus Z(p_{\lambda_0})$, we define
$$
\tilde{\Phi}(\lambda)=\frac{1}{p_{\lambda_0}(\lambda)}\big[ U(p_{\lambda_0}\otimes e_1)(\lambda),\cdots,(U(p_{\lambda_0}\otimes e_N))(\lambda)\big].
$$
For $f\in\mathcal{M}_1$, there is $\{q_i\}_{i=1}^\infty\subset \mathcal{P}_d\otimes\mathbb{C}^N$ such that $\lim\limits_{i\to \infty}q_i=f$ in the norm of $H_{\mathfrak{K}}\otimes \mathbb{C}^N$. Hence $\lim\limits_{i\to \infty}p_{\lambda_0}q_i=p_{\lambda_0}f$, which implies that
\begin{equation}\label{4.4.star1}
  \lim\limits_{i\to \infty} (U(p_{\lambda_0}q_i))(\lambda)=(U(p_{\lambda_0}f))(\lambda)=p_{\lambda_0}(\lambda)(Uf)(\lambda), ~\forall \lambda\in \Omega.
\end{equation}
On the other hand, write $q_i=\sum\limits_{i=1}^{N}q_{i,j}\otimes e_j$, then from
$$
U(p_{\lambda_0}q_i)=U\left(p_{\lambda_0}\sum_{j=1}^N q_{i,j}\otimes e_j\right)=\sum_{j=1}^N q_{i,j} U(p_{\lambda_0}\otimes e_j),
$$
we have for $\lambda\in\Omega\setminus Z(p_{\lambda_0})$,
\begin{equation}\label{4.4.star2}
 \lim\limits_{i\to \infty} (U(p_{\lambda_0}q_i))(\lambda)=\lim\limits_{i\to \infty}\sum_{j=1}^N q_{i,j}(\lambda) (U(p_{\lambda_0}\otimes e_j))(\lambda)=\lim\limits_{i\to \infty}p_{\lambda_0}(\lambda)\tilde{\Phi}(\lambda)q_i(\lambda)=p_{\lambda_0}(\lambda)\tilde{\Phi}(\lambda)f(\lambda).
\end{equation}
Then by \eqref{4.4.star1} and \eqref{4.4.star2},
\begin{equation}\label{4.4.star3}
  (Uf)(\lambda)=\tilde{\Phi}(\lambda)f(\lambda),~\forall \lambda\in \Omega\setminus Z(p_{\lambda_0}),~f\in \mathcal{M}_1.
\end{equation}

For fixed $\mu_0\in \Omega\setminus \mathbf{Z}$, there exist $f_1,\cdots,f_N\in \mathcal{M}_1$ such that $\operatorname{span}\{f_i(\mu_0):~1\leq i\leq N\}=\mathbb{C}^N$. Hence there exists a neighborhood $V$ of $\mu_0$ contained in $\Omega\setminus \mathbf{Z}$ such that 
$$
\operatorname{span}\{f_i(\lambda):~1\leq i\leq N\}=\mathbb{C}^N
$$ 
for $\lambda\in V$, which implies that
$$
F_V(\lambda)=[f_1(\lambda),\cdots,f_N(\lambda)]
$$
is invertible in $V$. For $\lambda\in V$, write $H_V(\lambda)=[(Uf_1)(\lambda),\cdots,(Uf_N)(\lambda)]$ and $\Phi_V(\lambda)=H_V(\lambda)F_V^{-1}(\lambda)$. By \eqref{4.4.star3}, for $\lambda\in V\cap (\Omega\setminus Z(p_{\lambda_0}))$ we have
$$
\begin{aligned}
\Phi_V(\lambda)
&=H_V(\lambda)F_V^{-1}(\lambda)\\
&=[(Uf_1)(\lambda),\cdots,(Uf_N)(\lambda)]F_V^{-1}(\lambda)\\
&=[\tilde{\Phi}(\lambda)f_1(\lambda),\cdots, \tilde{\Phi}(\lambda)f_N(\lambda)]F_V^{-1}(\lambda)\\
&=\tilde{\Phi}(\lambda)F_V(\lambda)F_V^{-1}(\lambda)\\
&=\tilde{\Phi}(\lambda).
\end{aligned}
$$
Hence $\tilde{\Phi}$ can be extended to a holomorphic function
\[
\Phi : \Omega\setminus \mathbf{Z} \to M_N(\mathbb{C})
\]
satisfying $(Uf)(\lambda) = \Phi(\lambda) f(\lambda)$ for all $\lambda \in \Omega\setminus \mathbf{Z} $ and $f \in \mathcal{M}_1$. 
Applying the identical argument to the inverse isomorphism $U^{-1} : \mathcal{M}_2 \to \mathcal{M}_1$
yields a holomorphic matrix-valued function $\Psi : \Omega\setminus \mathbf{Z} \to M_N(\mathbb{C})$ such that
$(U^{-1}g)(\lambda) = \Psi(\lambda) g(\lambda)$ for all $\lambda \in \Omega\setminus \mathbf{Z}$ and $g \in \mathcal{M}_2$. Moreover, for any $f \in \mathcal{M}_1$,
\begin{equation}\label{u.c}
f(\lambda) = (U^{-1} U f)(\lambda) = \Psi(\lambda) \Phi(\lambda) f(\lambda), \quad \lambda \in \Omega\setminus \mathbf{Z}.
\end{equation}
For $\lambda\in \Omega\setminus \mathbf{Z}$, we have $\{f(\lambda)|~f\in \mathcal{M}_1\}=\mathbb{C}^N$. Then from \eqref{u.c}, $\Psi(\lambda)\Phi(\lambda) = I_{\mathbb{C}^N}$ for $\lambda\in \Omega\setminus \mathbf{Z}$. A similar argument to $g \in \mathcal{M}_2$ yields $\Phi(\lambda) \Psi(\lambda) = I_{\mathbb{C}^N}$ for $\lambda\in \Omega\setminus \mathbf{Z}$.
\end{proof}
\begin{lem}\label{lem:finite-dimensional-estimate}
Let \(L\subset H_\mathfrak{K}\otimes\mathbb C^N\) be a finite-dimensional subspace. Then
\[
        \|P_L(\mathfrak{K}_\lambda\otimes y)\|^2
        =
        o(\mathfrak{K}(\lambda,\lambda))\|y\|^2,
        \qquad \lambda\to\partial\Omega,
        \tag{4.3}
\]
uniformly for \(y\in\mathbb C^N\).
\end{lem}

\begin{proof}
Let \(u_1,\ldots,u_q\) be an orthonormal basis of \(L\). Write
\[
        u_r=(u_{r,1},\ldots,u_{r,N}),
        \qquad u_{r,j}\in H_\mathfrak{K}.
\]
By the assumption (1)(2)(3), it is easy to see that $\frac{\mathfrak{K}_\lambda}{\|\mathfrak{K}_\lambda\|}$ converge weakly to \(0\) in \(H_\mathfrak{K}\) as \(\lambda\to\partial\Omega\). Hence for $1\leq r\leq q$, 
\begin{equation}\label{4.5star1}
        \frac{\|u_r(\lambda)\|_{\mathbb C^N}^2}
             {{\mathfrak{K}(\lambda,\lambda)}}
        =\sum_{j=1}^{N}
        \frac{|u_{r,j}(\lambda)|^2}{{\mathfrak{K}(\lambda,\lambda)}}
        =\sum_{j=1}^{N}
        |\langle u_{r,j},\mathfrak{k}_\lambda\rangle|^2
        \longrightarrow0.
\end{equation}
For \(y\in\mathbb C^N\), 
\begin{equation}\label{4.5star2}
        \|P_L(\mathfrak{K}_\lambda\otimes y)\|^2
        =
        \sum_{r=1}^q
        |\langle \mathfrak{K}_\lambda\otimes y,u_r\rangle|^2        
        =
        \sum_{r=1}^q
        |\langle y,u_r(\lambda)\rangle_{\mathbb C^N}|^2  
        \le
        \|y\|^2\sum_{r=1}^q\|u_r(\lambda)\|_{\mathbb C^N}^2 .
\end{equation}
By \eqref{4.5star1},
\[
        \sum_{r=1}^q\|u_r(\lambda)\|_{\mathbb C^N}^2
        =
        o(\mathfrak{K}(\lambda,\lambda)),
\]
then by \eqref{4.5star2}, we obtain
\[
        \|P_L(\mathfrak{K}_\lambda\otimes y)\|^2
        =
        o(\mathfrak{K}(\lambda,\lambda))\|y\|^2,
\]
uniformly for \(y\in\mathbb C^N\).
\end{proof}

\begin{proof}[Proof of Theorem~\ref{thm:vector-valued-rigidity}]
Let
\[
        \pi:H_\mathfrak{K}\otimes\mathbb C^N
        \longrightarrow
        (H_\mathfrak{K}\otimes\mathbb C^N)/\mathcal{M}_1
\]
be the quotient map. Since \(\mathcal{P}_d\otimes\mathbb C^N\) is dense in
\(H_\mathfrak{K}\otimes\mathbb C^N\), 
$
        \pi\bigl(\mathcal{P}_d\otimes\mathbb C^N\bigr)
$
is dense in the finite-dimensional space
$
        (H_\mathfrak{K}\otimes\mathbb C^N)/\mathcal{M}_1.
$
Hence there exist \(v_1,\ldots,v_s\in\mathcal{P}_d\otimes\mathbb{C}^N\) such that
$
        \pi(v_1),\ldots,\pi(v_s)
$
form a basis of \((H_\mathfrak{K}\otimes\mathbb C^N)/\mathcal{M}_1\). For \(f\in \mathcal{M}_1\), there exist
polynomials $\{q_n\}_{n=1}^\infty\subset \mathcal{P}_d\otimes\mathbb C^N$ such that
$
        \lim\limits_{n\to\infty}q_n=f.
$
Write
\[
        \pi(q_n)=\sum_{j=1}^s a_{n,j}\pi(v_j),
\]
then $ q_n-\sum\limits_{j=1}^s a_{n,j}v_j\in \mathcal{M}_1\cap(\mathcal{P}_d\otimes\mathbb C^N)$. Since $\lim\limits_{n\to\infty}\pi(q_n)=\pi(f)=0$, we have $\lim\limits_{n\to\infty}a_{n,j}=0$ for every $1\leq j\leq s$. Hence
\[
        \lim\limits_{n\to\infty}\left(q_n-\sum_{j=1}^s a_{n,j}v_j\right)=\lim\limits_{n\to\infty}q_n= f.
\]
Thus $\mathcal{M}_1$ is polynomially generated. By the same reasoning $\mathcal{M}_2$ is also polynomially generated.

By Lemma \ref{prop:scheme_vector_free_off_support}, $\mathcal{Z}(\mathcal{M}_i)=Z(\operatorname{Ann}_{\mathcal{P}_d}(\mathcal{Q}_i))\cap\Omega$, where $\mathcal{Q}_i=(\mathcal{P}_d\otimes\mathbb{C}^N)/(\mathcal{M}_i\cap\mathcal{P}_d\otimes\mathbb{C}^N)$.  Let
$$
\begin{aligned}
\phi_i:\mathcal{P}_d&\to (H_{\mathfrak{K}}/\mathcal{M}_i)^N\\
p&\mapsto ([p\otimes e_1],\cdots,[p\otimes e_N])
\end{aligned}
$$
Then $\ker \phi_i=\operatorname{Ann}_{\mathcal{P}_d}(\mathcal{Q}_i)$, which implies that
 $$
 \operatorname{dim} \mathcal{P}_d/\operatorname{Ann}_{\mathcal{P}_d}(\mathcal{Q}_i)=\dim\operatorname{ran}\phi_i\leq N\dim(H_{\mathfrak{K}}/\mathcal{M}_i)<\infty.
 $$
 Hence $\mathcal{Z}(\mathcal{M}_i)=Z(\operatorname{Ann}_{\mathcal{P}_d}(\mathcal{Q}_i))\cap\Omega$ are finite for $i=1,2$.
Let $\mathbf{Z}=\mathcal{Z}(\mathcal{M}_1)\cup\mathcal{Z}(\mathcal{M}_2)$, by Lemma~\ref{prop:vector_local_multiplier_unified}, there exist holomorphic functions
\[
        \Phi,\Psi:\Omega\setminus \mathbf{Z}\to M_N(\mathbb C)
\]
such that
\begin{equation}\label{4.10}
        (Uf)(\lambda)=\Phi(\lambda)f(\lambda),
        ~\forall f\in \mathcal{M}_1,~ \lambda\in\Omega\setminus \mathbf{Z},
\end{equation}
and
\[
        (U^{-1}g)(\lambda)=\Psi(\lambda)g(\lambda),
        ~\forall g\in \mathcal{M}_2,~ \lambda\in\Omega\setminus \mathbf{Z}.
\]
Moreover,
\begin{equation}\label{4.12}
        \Psi(\lambda)\Phi(\lambda)=I_N=\Phi(\lambda)\Psi(\lambda),
        ~\forall \lambda\in\Omega\setminus \mathbf{Z}.
\end{equation}
Since \(d\ge2\) and \(\mathbf{Z}\) is finite, Hartogs' removable singularity theorem implies that
\(\Phi\) and \(\Psi\) extend holomorphically to all of \(\Omega\). The identities in (4.12)
then extend to all of \(\Omega\). Hence
\[
        \Phi\in\mathcal O(\Omega,GL_N(\mathbb C)),
        \qquad
        \Psi(\lambda)=\Phi(\lambda)^{-1}.
\]

By Lemma~\ref{lem:finite-dimensional-estimate}, there exist
$
        \varepsilon_i(\lambda)\to0
$
as
\( \lambda\to\partial\Omega, \)
such that
\[
        \|P_{\mathcal{M}_i^\perp}(\mathfrak{K}_\lambda\otimes y)\|^2
        \le
        \varepsilon_i(\lambda)\mathfrak{K}(\lambda,\lambda)\|y\|^2,
        ~\forall y\in\mathbb C^N.
\]
Since
\[
\begin{aligned}
        \|P_{\mathcal{M}_i}(\mathfrak{K}_\lambda\otimes y)\|^2
        &=
        \|\mathfrak{K}_\lambda\otimes y\|^2
        -
        \|P_{\mathcal{M}_i^\perp}(\mathfrak{K}_\lambda\otimes y)\|^2,
\end{aligned}
\]
and $\|\mathfrak{K}_\lambda\otimes y\|^2=\mathfrak{K}(\lambda,\lambda)\|y\|^2$, 
we have
\begin{equation}\label{eq:Mi-estimate}
         (1-\varepsilon_i(\lambda))\mathfrak{K}(\lambda,\lambda)\|y\|^2
        \le
        \|P_{\mathcal{M}_i}(\mathfrak{K}_\lambda\otimes y)\|^2
        \le
        \mathfrak{K}(\lambda,\lambda)\|y\|^2, ~i=1,2.
\end{equation}
Fix \(\lambda\in\Omega\setminus \mathbf{Z}\) and \(x\in\mathbb C^N\). For \(f\in \mathcal{M}_1\), using
\eqref{4.10}, we get
\[
        \langle f,U^*P_{\mathcal{M}_2}(\mathfrak{K}_\lambda\otimes x)\rangle
        =
        \langle (Uf)(\lambda),x\rangle_{\mathbb C^N}                             
        =
        \langle \Phi(\lambda)f(\lambda),x\rangle_{\mathbb C^N}                   
        =
        \langle f,P_{\mathcal{M}_1}(\mathfrak{K}_\lambda\otimes\Phi(\lambda)^*x)\rangle ,
\]
which implies that
\[
    \|P_{\mathcal{M}_2}(\mathfrak{K}_\lambda\otimes x)\|
    =
    \|P_{\mathcal{M}_1}(\mathfrak{K}_\lambda\otimes\Phi(\lambda)^*x)\|.
\]
Take $y=\Phi(\lambda)^*x$, by \eqref{eq:Mi-estimate} we obtain
\[
    (1-\varepsilon_1(\lambda))\mathfrak{K}(\lambda,\lambda)\|\Phi(\lambda)^*x\|^2
    \le
    \|P_{\mathcal{M}_1}(\mathfrak{K}_\lambda\otimes\Phi(\lambda)^*x)\|^2
    =
    \|P_{\mathcal{M}_2}(\mathfrak{K}_\lambda\otimes x)\|^2
    \le
    \mathfrak{K}(\lambda,\lambda)\|x\|^2.
\]
which implies that
\[
    \|\Phi(\lambda)^*x\|^2 \le \frac{1}{1-\varepsilon_1(\lambda)}\|x\|^2,
    \qquad \forall\,x\in\mathbb C^N.
\]
Since $\varepsilon_1(\lambda)\to0$ as $\lambda\to\partial\Omega$, it follows that
\begin{equation}\label{eq:Phi-bd}
    \limsup_{\lambda\to\partial\Omega}\|\Phi(\lambda)\|\le1.
\end{equation}
Applying the same argument to the unitary module isomorphism
$U^{-1}:\mathcal{M}_2\to\mathcal{M}_1$ yields
\begin{equation}\label{eq:Phi-inv-bd}
    \limsup_{\lambda\to\partial\Omega}\|\Phi(\lambda)^{-1}\|
    =\limsup_{\lambda\to\partial\Omega}\|\Psi(\lambda)\|
    \le1.
\end{equation}
For each unit vector \(v\in\mathbb C^N\), the function
$
    \lambda\longmapsto \|\Phi(\lambda)v\|^2
$
is plurisubharmonic on \(\Omega\). By \eqref{eq:Phi-bd} and the maximum principle for plurisubharmonic functions,
\[
    \|\Phi(\lambda)v\|\le1,\qquad \forall\,\lambda\in\Omega,\ \|v\|=1,
\]
hence \(\|\Phi(\lambda)\|\le1\) for all \(\lambda\in\Omega\). Applying the same argument to \(\Phi^{-1}\) and using \eqref{eq:Phi-inv-bd} yields \(\|\Phi(\lambda)^{-1}\|\le1\) on \(\Omega\). 
Now, for any \(\lambda\in\Omega\) and any \(x\in\mathbb C^N\),
\[
    \|x\| = \|\Phi(\lambda)^{-1}\Phi(\lambda)x\|
    \le \|\Phi(\lambda)x\|
    \le \|x\|,
\]
which forces \(\|\Phi(\lambda)x\| = \|x\|\). Therefore \(\Phi(\lambda)\) is unitary for every \(\lambda\in\Omega\):
\[
    \Phi(\lambda)^*\Phi(\lambda)=I_N,\qquad \forall\,\lambda\in\Omega.
\]
Differentiating this identity with respect to \(z_j\) and using the fact that \(\Phi^*\) is anti-holomorphic, we get
$
    \Phi(\lambda)^* \frac{\partial\Phi}{\partial z_j}(\lambda)=0.
$
Since \(\Phi(\lambda)\) is unitary, \(\frac{\partial\Phi}{\partial z_j}(\lambda)=0\) for all \(j\). Hence there exists a unitary matrix $W$ such that 
\(\Phi(\lambda)\equiv W\). Then by \eqref{4.10},
\[
        (Uf)(\lambda)=Wf(\lambda),
        \ \forall\,f\in \mathcal{M}_1,~\lambda\in\Omega\setminus \mathbf{Z}.
\]
Since both sides are holomorphic, the equality holds on all of \(\Omega\). Hence
\[
        U=(I_{H_\mathfrak{K}}\otimes W)|_{\mathcal{M}_1}.
\]
Because \(U(\mathcal{M}_1)=\mathcal{M}_2\), we conclude that
\[
        \mathcal{M}_2=(I_{H_\mathfrak{K}}\otimes W)\mathcal{M}_1.
\]
\end{proof}

\section{The finite defect problem: infinitely-many-variable case}

The finite defect problems also make sense for unitarily invariant CNP kernel space $H_{\mathbb{K}}$ in infinitely many variables, where $\mathbb{K}$ is a kernel on the unit ball $\mathbb B$ in a separable infinite-dimensional Hilbert space $H$ defined by
\begin{equation}\label{k inf}
  \mathbb{K}:\mathbb{B} \times \mathbb{B} \rightarrow \mathbb{C}; ~ \mathbb{K}(z,w)=\sum_{n=0}^{\infty}a_n \left\langle z,w \right\rangle^{n},
\end{equation}
with $a_0 = 1$ and $a_n > 0$ for all $n \geq 1$. Moreover, we assume that $\mathbb K$ satisfies the regularity condition(I)-(III).  It is easy to compute that for $\alpha=(\alpha_1,\cdots,\alpha_n,0,0\cdots)\in\mathbb{N}^\mathbb{N}$,
$$
\|z^\alpha\|^2=\frac{\alpha!}{a_{|\alpha|}|\alpha|!}
$$
and hence 
\[
M_{z_i}^* z^{\alpha} =
\begin{cases}
0, & \text{if } \alpha_i = 0, \\
\frac{a_{|\alpha|-1}\alpha_i}{a_{|\alpha|}(\alpha_1 + \cdots + \alpha_n)} z_1^{\alpha_1} \cdots z_i^{\alpha_i - 1} \cdots z_n^{\alpha_n}, & \text{if } \alpha_i > 0.
\end{cases}
\]

Earlier in this paper, we solved the finite defect problem for regular unitarily invariant CNP reproducing kernel Hilbert space of finitely many variables. For infinitely many variables, however, the situation is different. In a previous article \cite{curv}, we obtained the following result.
\begin{thm}\cite{curv}
Let $H^2$ be the Drury-Arveson space of infinitely many variables,  $\cal M$ be a nonzero submodule in $H^2$. If $\cal M$ is finite defect, then $\mathcal{M}=H^2$. 
\end{thm}
Similar to the finitely-many-variable case, by direct computation we have
\begin{equation}\label{defect operator inf}
\Delta_{\mathcal{M}}^{2}=P_{\mathcal{M}}\left(E_0+\sum_{\alpha\in\mathbb{N}_+^{(\mathbb{N})}}b_\alpha[M_{z^\alpha},P_{\mathcal{M}^{\perp}}]
[P_{\mathcal{M}^{\perp}},M_{z^\alpha}^{*}]\right)P_{\mathcal{M}},
\end{equation}
In the present section, we obtain a stronger result on a general framework. To do this, we need the following lemma.
\begin{lem}\label{PM1 neq 0}
  Let $\mathcal{M}$ be a submodule of $H_{\mathbb{K}}\otimes \mathbb{C}^N$. If there exist $\eta\in \mathbb{C}^N$ and $h\in H_{\mathbb{K}}$ such that
  $$
  P_\mathcal{M}(1\otimes\eta)=0~\text{and}~P_\mathcal{M}(h\otimes\eta)\neq 0,
  $$
  then $\Delta_\mathcal{M}$ is not compact.
\end{lem}
\begin{proof}
  Without loss of generality, suppose $\eta$ is a unit vector.
  Since $$P_\mathcal{M}(1\otimes\eta)=0\text{ and }P_\mathcal{M}(h\otimes\eta)\neq 0,$$ obviously there is an integer $r\geq 0$, such that for $0\leq i\leq r$
  \begin{equation}\label{Pm Hr+1 neq 0}
    P_\mathcal{M}(H_i\otimes\mathbb{C}_\eta)=\{0\}~\text{and}~P_\mathcal{M}(H_{r+1}\otimes\mathbb{C}_\eta)\not=\{0\},
  \end{equation}
  where $\mathbb{C}_\eta$ is defined as \eqref{Ceta}.
  Next, the lemma will be proved in two cases. \setcounter{Case}{0} 
 \begin{Case}
  $\exists \alpha_0\in \mathbb{N}^{(\mathbb{N})}$ such that $|\alpha_0|=r$ and $\limsup\limits_{i\rightarrow \infty}\|P_\mathcal{M}(z_iz^{\alpha_0}\otimes\eta)\|>0$.
  \end{Case}
  Obviously there is an $\varepsilon>0$ and subsequence $\{n_i\}_{i\geq 1}^\infty$ such that
  $$
  \|P_\mathcal{M}(z_{n_i}z^{\alpha_0}\otimes\eta)\|>\varepsilon,~\forall i\geq 1.
  $$
  Please note that $\forall \alpha\in \mathbb{N}^{(\mathbb{N})}\setminus\{0\}$, $deg(M_{z^\alpha}^*z_{n_i}z^{\alpha_0})\leq r$. By \eqref{Pm Hr+1 neq 0},
  $$
  P_\mathcal{M}\left((M_{z^\alpha}\otimes I)^*z_{n_i}z^{\alpha_0}\otimes\eta\right)=0.
  $$
  It follows that
   $$
  \begin{aligned}
  &\|\Delta_\mathcal{M}^2P_\mathcal{M}(z_{n_i}z^{\alpha_0}\otimes\eta)\|\\
  &=\left\|P_\mathcal{M}(z_{n_i}z^{\alpha_0}\otimes\eta)-\sum_{\alpha\in \mathbb{N}_+^{(\mathbb{N})}}b_\alpha (M_{z^\alpha}\otimes I)P_\mathcal{M}(M_{z^\alpha}^*\otimes I)P_\mathcal{M}(z_{n_i}z^{\alpha_0}\otimes\eta)\right\|\\
  &=\|P_\mathcal{M}(z_{n_i}z^{\alpha_0}\otimes\eta)\|>\varepsilon.
  \end{aligned}
  $$
   Notice that $z_{n_i}z^{\alpha_0}\otimes\eta\overset{w}\rightharpoonup 0$, and then $\Delta_{\mathcal{M}}$ is not compact.
   \begin{Case}
  For all $|\alpha|=r$, $\lim\limits_{i\rightarrow\infty}\|P_\mathcal{M}(z_{i}z^\alpha\otimes\eta)\|=0$.
  \end{Case}
  Since $P_{H_{r+1}\otimes \mathbb{C}_\eta}{\cal M}\neq 0$, there is an $f\in\mathcal{M}$ such that $P_{H_{r+1}\otimes\mathbb{C}_\eta}f\neq 0$. It follows that, there is an $\alpha_0\in\mathbb{N}^{(\mathbb{N})}$ with $|\alpha_0|=r+1$, such that
  $$
  \left\langle f,z^{\alpha_0}\otimes\eta\right\rangle\neq 0.
  $$
  Write $z^{\alpha_0}=z^{\beta_0}z_k$ for some $k\geq 1$ and $|\beta_0|=r$. Next, we claim that 
 $ P_{\mathcal{M}^\perp}(M_{z_k}^*\otimes I)P_{\mathcal{M}}$
  is not compact. In fact,
  $$
  \begin{aligned}
  &\limsup\limits_{i\rightarrow\infty}\|P_{\mathcal{M}^\perp}(M_{z_k}^*\otimes I)P_\mathcal{M}(z_if)\|\\
  &\geq \limsup\limits_{i\rightarrow\infty}
  \left|\left\langle P_{\mathcal{M}^\perp}(M_{z_k}^*\otimes I) (z_if),\frac{z_iz^{\beta_0}\otimes\eta}{\|z_iz^{\beta_0}\|}\right\rangle\right|\\
  &= \limsup\limits_{i\rightarrow\infty}
  \left|\left\langle (M_{z_k}^*\otimes I) (z_if),\frac{z_iz^{\beta_0}\otimes\eta}{\|z_iz^{\beta_0}\|}\right\rangle\right.
  \left.-\left\langle (M_{z_k}^*\otimes I) (z_if),P_\mathcal{M}\frac{z_iz^{\beta_0}\otimes\eta}{\|z_iz^{\beta_0}\|}\right\rangle\right|\\
  &=\limsup\limits_{i\rightarrow\infty}
  \left|\left\langle f,\frac{(M_{z_i}^*z_iz^{\alpha_0})\otimes\eta}{\|z_iz^{\beta_0}\|}\right\rangle\right|\\
  &=C\left|\left\langle f,z^{\alpha_0}\otimes\eta\right\rangle\right|>0
  \end{aligned}
  $$
  where $C=\frac{a_{r+1}}{a_{r+2}\,(r+2)}\sqrt{\frac{a_{r+1}\,(r+1)!}{\beta!}}$ is a constant, which does not depend on $i$.   Therefore, $P_{\mathcal{M}^\perp}(M_{z_k}^*\otimes I)P_{\mathcal{M}}$ is not compact, and hence by \eqref{defect operator inf}, $\Delta_\mathcal{M}$ is not compact.
\end{proof}
The following theorem is the main result in this section.
\begin{thm}\label{finite defect for NP}
  Let $\mathcal{M}$ be a submodule of $H_{\mathbb{K}}\otimes \mathbb{C}^N$. Then the following statements are equivalent.
  \begin{itemize}
    \item[(1)] $\Delta_{\mathcal{M}}$ is compact,
    \item[(2)] there is a subspace $F\subset\mathbb{C}^N$ such that $\mathcal{M}=H_{\mathbb{K}}\otimes F$.
  \end{itemize}
\end{thm}

\begin{proof}
$(2)\Rightarrow(1)$ is obvious. $(1)\Rightarrow(2)$: Let
\begin{equation}\label{F inf}
F=\{\zeta\in \mathbb{C}^N|1\otimes \zeta \in \mathcal{M}\}.
\end{equation}
By (\ref{defect operator inf}), it is easy to see that for $\zeta\in F$, 
\begin{equation}\label{delta 1 = 1}
  \Delta_\mathcal{M}^2 (1\otimes\zeta)=1\otimes\zeta.
\end{equation}
Obviously $H_{\mathbb{K}}\otimes F \subseteq \mathcal{M}$. Hence to prove $H_{\mathbb{K}}\otimes F=\mathcal{M}$, it suffices to prove that
$
H_{\mathbb{K}}\otimes F^\perp \subseteq \mathcal{M}^\perp.
$
Suppose that $H_{\mathbb{K}}\otimes F^\perp \nsubseteq \mathcal{M}^\perp$, then there exists a unit vector $\eta\in F^\perp$ and $z^\alpha\in H_{\mathbb{K}} $ such that $z^\alpha \otimes \eta\notin \mathcal{M}^\perp$. Hence by Lemma \ref{PM1 neq 0}, we have 
\begin{equation}\label{PM1 eta neq 0}
  P_\mathcal{M} (1\otimes \eta)\neq 0.
\end{equation}
Write 
\begin{equation}\label{Pm1oeta=c+f}
P_\mathcal{M} (1\otimes\eta)=1\otimes\zeta+ \sum_{m=1}^{\infty}\sum_{|\alpha|=m}z^\alpha\otimes c_{\alpha}, ~\text{where}~ \zeta \in \mathbb{C}^N~ \text{and}~c_\alpha \in \mathbb{C}^N.
\end{equation}
Notice that $P_\mathcal{M} (z_i\otimes\eta)\rightharpoonup 0$. Hence 
\begin{equation}\label{PMzieta=0}
\lim\limits_{i\rightarrow\infty}\Delta_\mathcal{M}^2P_\mathcal{M} (z_i\otimes\eta)=0.
\end{equation}
By \eqref{k} and \eqref{1.a intro}, it is easy to see that
\begin{equation}\label{a1=b1}
  a_1=b_1.
\end{equation}
Hence for $i\geq1,$
\begin{equation}\label{haode}
\begin{aligned}
\Delta_\mathcal{M}^2 P_\mathcal{M}(z_i\otimes \eta)&=\left(P_\mathcal{M}-\sum \limits_{\alpha\in \mathbb{N}_+^{(\mathbb{N})}}b_\alpha (M_{z^\alpha}\otimes I) P_{\mathcal{M}} (M_{z^\alpha}^*\otimes I)\right)P_\mathcal{M}(z_i\otimes\eta)\\
&=P_\mathcal{M}(z_i\otimes\eta)-z_iP_\mathcal{M} (1\otimes\eta).
\end{aligned}
\end{equation}
Notice that
\begin{equation}\label{product with z_iPm1}
  \left\langle P_\mathcal{M}(z_i\otimes\eta)-z_iP_\mathcal{M} (1\otimes\eta),z_iP_\mathcal{M} (1\otimes\eta)\right\rangle
  =\frac{1}{a_1}\|P_\mathcal{M} (1\otimes\eta)\|^2-\|z_iP_\mathcal{M} (1\otimes\eta)\|^2.
\end{equation}
By \eqref{PMzieta=0} and \eqref{haode} we have
\begin{equation}\label{z_ipm=pm}
  \lim\limits_{i\rightarrow\infty}\|z_iP_\mathcal{M} (1\otimes\eta)\|^2=\frac{1}{a_1}\|P_\mathcal{M} (1\otimes\eta)\|^2.
\end{equation}
From \eqref{Pm1oeta=c+f}, by direct computation we have
\begin{equation}\label{am+1}
\begin{aligned}
\lim\limits_{i\rightarrow\infty}\|z_iP_\mathcal{M} (1\otimes\eta)\|^2=&\lim\limits_{i\rightarrow\infty}
\left(\|z_i\|^2\|\zeta\|^2+\sum_{m=1}^{\infty}\sum_{|\alpha|=m}\|z_iz^\alpha\otimes c_{\alpha}\|^2\right)\\
=&\frac{\|\zeta\|^2}{a_1}+\sum_{m=1}^{\infty}\sum_{|\alpha|=m}\frac{\alpha!}{a_{m+1}(m+1)!}\|c_\alpha\|^2,
\end{aligned}
\end{equation}
and
\begin{equation}\label{a1am}
\begin{aligned}
\frac{1}{a_1}\|P_\mathcal{M} (1\otimes\eta)\|^2&=\frac{1}{a_1}\left(\|\zeta\|^2+\sum_{m=1}^{\infty}\sum_{|\alpha|=m}\|z^\alpha\otimes c_{\alpha}\|^2\right)\\
&=\frac{\|\zeta\|^2}{a_1}+\sum_{m=1}^{\infty}\sum_{|\alpha|=m}\frac{\alpha!}{a_1a_{m}m!}\|c_\alpha\|^2.
\end{aligned}
\end{equation}
Since $a_1a_m\leq a_{m+1}$\cite[Corollary 6.6]{HM}, for all $\alpha\in \mathbb{N}^\mathbb{N}\setminus\{0\}$,
\begin{equation}\label{a_m+1m+1>a_1a_m}
\frac{\alpha!}{a_{m+1}(m+1)!}<\frac{\alpha!}{a_1a_{m}m!}.
\end{equation}
Then by \eqref{z_ipm=pm}, \eqref{am+1}, \eqref{a1am} and \eqref{a_m+1m+1>a_1a_m},
$$
c_\alpha=0,~\forall \alpha\in \mathbb{N}^\mathbb{N}\setminus\{0\}.
$$
Hence $P_\mathcal{M}(1\otimes\eta)=1\otimes\zeta$, which implies that $\zeta\in F$. Therefore
$$
\|P_\mathcal{M}(1\otimes\eta)\|^2=\langle 1\otimes\eta,1\otimes\zeta\rangle=0,
$$
which contradicts \eqref{PM1 eta neq 0}. 
\end{proof}

\section{The Failure of essential normality of submodules: infinitely-many-variable case}

At the end of the present paper, we will consider the essential normality of submodules in $H_{\mathbb{K}}$. Geometric Arveson-Douglas Conjecture deals with the essential normality of submodules in Drury-Arveson module $H_d^2$, weighted Bergman space and etc, which is one of the most important topics in Functional Analysis, there are many efforts along this line, \cite{Arv1,DGW,DW,DWy,Eng,FX,GW,Ken,KS,Xia} and the references therein. In this infinitely-many-variable case, the answer to the geometric Arveson-Douglas conjecture is negative.

Let $\mathcal{H}$ be a separable infinite-dimensional Hilbert space and $\mathbb{B}$ the unit ball in $\mathcal{H}$. Let $H_{\mathbb{K}}$ be a regular unitarily invariant reproducing kernel Hilbert space (UIRKHS) of holomorphic functions on $\mathbb{B}$.
As usual write $R_{z_i}=M_{z_i}|_{\mathcal{M}}$. A submodule $\mathcal{M}$ of $H_{\mathbb{K}}$ is called essentially normal, if $R_{z_i}R_{z_i}^*-R_{z_i}^*R_{z_i}$ is compact for each $i\geq 1$. In this section, we prove the following theorem.
\begin{thm}\label{PM1=0}
   Let $\mathcal{M}$ be a nonzero submodule of $H_{\mathbb{K}}\otimes \mathbb{C}^N$. Then $\mathcal{M}$ is not essentially normal.
\end{thm}
\begin{proof}
Since $\lim\limits_{n\to \infty}\frac{a_n}{a_{n+1}}=1$, for $j\geq 1$, $\|M_{z_j}\|=\sup\limits_n \sqrt{\frac{a_n}{a_{n+1}}}<\infty$. For simplicity, denote 
\begin{equation}\label{5.a1}
s=\|M_{z_j}\|,
\end{equation}
which is independent of $j$. Let $E_{r,\infty}: H_{\mathbb{K}}\otimes \mathbb{C}^N\to H_{r}$ be the orthogonal projection with $H_r$ being defined as 
 $$
 H_{r}=\{f\in H_{\mathbb{K}}\otimes \mathbb{C}^N;  f\text{ is homogeneous and~} deg f=r\}.
 $$
From $\mathcal{M}\neq 0$, there is an integer $r\geq 0$ and a unit vector $f\in \mathcal{M}$ satisfying ${E}_{r,\infty}f\neq 0$. Let $H_{\mathbb{K},i}=\overline{\operatorname{span}\{z^\alpha|\alpha\in\mathbb{N}^i\}}$. Then there is an integer $k>0$ such that for $i\geq k-1$,
\begin{equation}\label{5.b}
\|P_{H_{\mathbb{K},i}\otimes \mathbb C^N}f-f\|<\frac{a_r}{8s^4a_{r+2}(r+1)(r+2)}\|{E}_{r,\infty}f\|^2.
\end{equation}
Write
$$
f_i=P_{H_{\mathbb{K},i}\otimes \mathbb C^N}f\text{ and }~f^{(i)}=f-f_i.
$$
Since $(M_{z_k}\otimes I)^*z_if_{k-1}=0$ for $i>k$, by \eqref{5.a1}, we have
\begin{equation}\label{5.a}
(M_{z_k}\otimes I)^*(z_if)=(M_{z_k}\otimes I)^*(z_if^{(k-1)}).
\end{equation}
Hence
$
  \|R_{z_k}R_{z_k}^*(z_if)\|\leq s^3\|f^{(k-1)}\|.
$
It follows that 
  \begin{equation}\label{5.3}
  \begin{aligned}
  \|(R_{z_k}^*R_{z_k}-R_{z_k}R_{z_k}^*)(z_if)\|\geq
    \frac{\|R_{z_k}(z_if)\|^2}{\|z_if\|}-s^3\|f^{(k-1)}\|.
   \end{aligned}
  \end{equation}
By \eqref{5.a1}, \eqref{5.b} and $\|f_{k-1}\|\leq \|f\|=1$, we have 
\begin{equation}\label{5.c}
\begin{aligned}
  \|z_kz_if_{k-1}\|\left\|z_kz_if^{(k-1)}\right\|
  \leq \frac{a_r}{8a_{r+2}(r+1)(r+2)}\|{E}_{r,\infty}f\|^2,
\end{aligned}
\end{equation}
which implies that 
   \begin{equation}\label{5.1}
     \begin{aligned}
    \frac{\|R_{z_k}(z_if)\|^2}{\|z_if\|}
   \geq&
   \frac{\left(\left\|z_kz_if_{k-1}\right\|-\left\|z_kz_if^{(k-1)}\right\|\right)^2}{s}\\
   \geq&s^{-1}
   \left(\|z_kz_if_{k-1}\|^2-2\|z_kz_if_{k-1}\|\left\|z_kz_if^{(k-1)}\right\|\right)\\
   \geq &s^{-1}\|E_{r+2,\infty}(z_kz_if_{k-1})\|^2-\frac{s^{-1}a_r}{4a_{r+2}(r+1)(r+2)}\|{E}_{r,\infty}f\|^2.
   \end{aligned}
   \end{equation}
  Thus, combining (\ref{5.3}), (\ref{5.1}) and (\ref{5.b}), 
   \begin{equation}\label{5.6}
     \begin{aligned}
     & \|(R_{z_k}^*R_{z_k}-R_{z_k}R_{z_k}^*)(z_if)\|\\
     \geq &s^{-1}\|E_{r+2,\infty}(z_kz_if_{k-1})\|^2-\frac{3a_r}{8sa_{r+2}(r+1)(r+2)}\|{E}_{r,\infty}f\|^2.
     \end{aligned}
   \end{equation}
 Next, we claim that
\begin{equation}\label{6.a}
\begin{aligned}
  \|E_{r+2,\infty}(z_kz_if_{k-1})\|^2\geq 
  \frac{a_{r}}{a_{r+2}(r+1)(r+2)}\|{E}_{r,\infty}f\|^2-\frac{a_r}{4a_{r+2}(r+1)(r+2)}\|{E}_{r,\infty}f\|^2.
\end{aligned}
\end{equation}
Indeed, if the claim \eqref{6.a} holds, then by \eqref{5.6},
$$
\begin{aligned}
\|(R_{z_k}^*R_{z_k}-R_{z_k}R_{z_k}^*)(z_if)\|
\geq \frac{3a_r}{8sa_{r+2}(r+1)(r+2)}\|{E}_{r,\infty}f\|^2
>0,
\end{aligned}
$$
and notice that $ z_if\overset{w}\rightharpoonup 0$, it follows that $R_{z_k}^*R_{z_k}-R_{z_k}R_{z_k}^*$ is not compact, i.e. $\mathcal{M}$ is not essentially normal.   
 We now proceed to prove the claim \eqref{6.a}.
By \eqref{5.b} and \eqref{5.a1}, 
\begin{equation}\label{5.e}
\begin{aligned}
  \|E_{r+2,\infty}(z_kz_if)\|\|E_{r+2,\infty}(z_kz_if^{(k-1)})\|=&
  \|z_kz_i{E}_{r,\infty}f\|\|E_{r+2,\infty}(z_kz_if^{(k-1)})\|
  \\\leq &\frac{a_r}{8a_{r+2}(r+1)(r+2)}\|{E}_{r,\infty}f\|^2.
\end{aligned}
\end{equation}
Since $\|z^\alpha\|^2=\frac{\alpha!}{a_{|\alpha|}|\alpha|!}$ for $\alpha \in \mathbb{N}^{\mathbb{N}}$, it easy to see that \begin{equation}\label{5.51}
\|(M_{z_j}\otimes I)g\|^2\geq \frac{a_l}{a_{l+1}(l+1)}\|g\|^2
\end{equation}
 for all $j$ and $g\in \overline{\operatorname{ran}E_{l,\infty}}$. Hence \eqref{5.e} and \eqref{5.51} ensure that
\begin{equation}\label{5.5}
  \begin{aligned}
  &\|E_{r+2,\infty}(z_kz_if_{k-1})\|^2\\
  \geq &\left(\|E_{r+2,\infty}(z_kz_if)\|-\|E_{r+2,\infty}(z_kz_if^{(k-1)})\|\right)^2\\
  \geq &\|z_kz_i{E}_{r,\infty}f\|^2-2\|E_{r+2,\infty}(z_kz_if)\|\|E_{r+2,\infty}(z_kz_if^{(k-1)})\|\\
  \geq &\frac{a_{r}}{a_{r+2}(r+1)(r+2)}\|{E}_{r,\infty}f\|^2-\frac{a_r}{4a_{r+2}(r+1)(r+2)}\|{E}_{r,\infty}f\|^2.
  \end{aligned}
\end{equation}
\end{proof}
For the Arveson space in infinitely many variables, we have shown that nontrivial submodules are never essentially normal. Nevertheless, there exist nontrivial quotient modules with one-dimensional zero variety that are indeed essentially normal, as demonstrated by the following example. For a submodule $\mathcal{M}\subset H^2$, write
$S_f=P_{\mathcal{M}^\perp}M_f|_{\mathcal{M}^\perp}$ for $f\in \operatorname{Mult}(H^2)$.

\begin{exam}\label{M not essentially normal}
  Let $\mathcal{M} =\left[z_{1}-\lambda z_{2}, z_{2}-\lambda z_{3}, \cdots\right]\subset H^2$, where $\lambda>1$. Then $M^{\perp}$ is essentially normal.
\end{exam}
\begin{proof}
Set
$
g_{0}=\sum\limits_{i=1}^{\infty} \lambda^{-i+1} z_{i}.
$
First, we claim that
\begin{equation}\label{g0n}
{\mathcal{M}}^\perp=\overline{\operatorname{span}\left\{1, {g}_{0}, {g}_{0}^{2}, \cdots\right\}}.
\end{equation}
In fact, it is easy to see that
\begin{equation}\label{g_0^n}
  g_0^n=\sum_{|\alpha| = n} \frac{n!}{\alpha!} \lambda^{n - \sum\limits_{k=1}^\infty k \alpha_k}  z^\alpha,\quad n\geq 1.
\end{equation}
Hence for $n\geq1, i\geq1, \beta \in \mathbb N^{\mathbb N},$ 
$$
\begin{aligned}
&\left\langle g_0^n, z^\beta (z_i-\lambda z_{i+1})\right\rangle
\\
=&\left\langle \frac{n!}{\beta!(\beta_i+1)}\lambda^{n -i- \sum\limits_{k=1}^\infty k \beta_k}z^\beta z_i, z^\beta z_i\right\rangle
-\left\langle \frac{n!}{\beta!(\beta_{i+1}+1)}\lambda^{n -(i+1)- \sum\limits_{k=1}^\infty k \beta_k}z^\beta z_{i+1}, \lambda z^\beta z_{i+1}\right\rangle\\
=&\frac{n!}{\beta!(\beta_{i}+1)}\lambda^{n -i- \sum\limits_{k=1}^\infty k \beta_k}\frac{\beta!(\beta_i+1)}{(|\beta|+1)!}
-\frac{n!}{\beta!(\beta_{i+1}+1)}\lambda^{n -i- \sum\limits_{k=1}^\infty k \beta_k}\frac{\beta!(\beta_{i+1}+1)}{(|\beta|+1)!}\\
=&0.
\end{aligned}
$$
Notice that $1\perp \mathcal{M}.$
Thus
\begin{equation}\label{span in Mperp}
\overline{\operatorname{span}\left\{1, {g}_{0}, {g}_{0}^{2}, \cdots\right\}}\perp \mathcal{M}.
\end{equation}
By a simple computation,
\begin{equation}\label{span in Mperp1}
z_1=\frac{\lambda^{2}-1}{\lambda^{2}}\left(g_{0}+\sum_{n=1}^{\infty}\left( \lambda^{-2 n} \sum_{k=1}^{n} \lambda^{k-1} (z_k-\lambda z_{k+1}) \right)\right).
\end{equation}
Then by \eqref{span in Mperp} and \eqref{span in Mperp1}, \begin{equation}\label{z_1 in oplus}z_1\in {\cal M}\oplus \overline{\operatorname{span}\left\{1, {g}_{0}, {g}_{0}^{2}, \cdots\right\}}.\end{equation}
Moreover, it is clear that for $n\geq 2$, 
\begin{equation}\label{z_i in oplus}
z_{n}=\lambda^{-n}\left(z_{1}-\sum\limits_{k=1}^{n-1} \lambda^{k-1} (z_k-\lambda z_{k+1}) \right). 
\end{equation}
Hence by \eqref{z_1 in oplus} and \eqref{z_i in oplus}, for $n\geq 1$, there exist nonzero $h_n\in \overline{\operatorname{span}\left\{1, {g}_{0}, {g}_{0}^{2}, \cdots\right\}}$
 and $f_n\in\mathcal{M}$ such that
\begin{equation}\label{zn=f+mu g}
  z_n=h_n+f_n. 
\end{equation}
Notice $g_0^i\bot g_0^j, i\neq j.$ Then there exist $\mu_n\neq 0$ such that $h_n=\mu_n g_0.$ 
Therefore \eqref{span in Mperp}, \eqref{zn=f+mu g} and $g_0\in \operatorname{Mult}(H^2)$ ensures that for $\alpha\in\mathbb{N}^\mathbb{N}$,
\begin{equation}\label{z alpha in oplus}
z^\alpha\in \mathcal{M} \oplus \overline{\operatorname{span}\left\{1, {g}_{0}, {g}_{0}^{2}, \cdots\right\}}.
\end{equation}
which implies that $\mathcal{M}^\perp=\overline{\operatorname{span}\left\{1, {g}_{0}, {g}_{0}^{2}, \cdots\right\}}$. The claim is proved.

 For $i\geq1,$ since $f_i\in \mathcal{M}$, we have $S_{f_{i}}=0$. Then by \eqref{zn=f+mu g}, $S_{z_{i}}=\mu_{i} S_{{g}_{0}},$ which implies that
\begin{equation}\label{Szi}
S_{z_{i}} S_{z_{i}}^{*}-S_{z_{i}}^{*} S_{z_{i}}=\left|\mu_{i}\right|^{2}\left(S_{{g}_{0}} S_{{{g}}_{0}}^{*}-S_{{{g}}_{0}}^{*} S_{{{g}}_{0}}\right).
\end{equation}
Hence to prove $\mathcal{M}^\perp$ is essentially normal, it suffices to prove that $S_{{g}_{0}} S_{{{g}}_{0}}^{*}-S_{{{g}}_{0}}^{*} S_{{{g}}_{0}}$ is compact. From
$$
\left\langle S_{{{g}}_{0}}^{*}g_0^m,g_0^n \right\rangle=\left\langle g_0^m,g_0^{n+1} \right\rangle=\delta_{m,n+1}\|g_0^{m}\|^2,~\forall m,n\geq 0,
$$
and \eqref{g0n}, we obtain that
\begin{equation}\label{Sg0*}
S_{{{g}}_{0}}^{*}g_0^m=\frac{\|g_0^m\|^2}{\|g_0^{m-1}\|^2}g_0^{m-1},~\forall m\geq 1
\end{equation}
and $S_{{{g}}_{0}}^{*}1=0$. By \eqref{g_0^n}, $\|g_0^n\|^2=\left(\frac{\lambda^2}{\lambda^2-1}\right)^n, n\geq1;$ therefore $\frac{\|g_0^n\|^2}{\|g_0^{n-1}\|^2}-\frac{\|g_0^{n+1}\|^2}{\|g_0^{n}\|^2}=0.$ It follows that for $n\geq 1$, \eqref{Sg0*} gives that
$$
\left(S_{{g}_{0}} S_{{{g}}_{0}}^{*}-S_{{{g}}_{0}}^{*} S_{{{g}}_{0}}\right)g_0^n
=\left(\frac{\|g_0^n\|^2}{\|g_0^{n-1}\|^2}-\frac{\|g_0^{n+1}\|^2}{\|g_0^{n}\|^2}\right)g_0^n
=0,
$$
which implies that $S_{{g}_{0}} S_{{{g}}_{0}}^{*}-S_{{{g}}_{0}}^{*} S_{{{g}}_{0}}$ is compact by \eqref{g0n}.
\end{proof}

\begin{rem}
  It is worth noting that $S_{{{g}}_{0}}$ is equal to a multiple of the shift operator on the Hardy space $H^2(\mathbb{D})$. 
\end{rem}
Xia \cite{Xia} proved that over the Arveson space in finitely many variables, for quotient modules defined by $V$, the commutators of the multiplication operators lie in the Schatten class \(\mathcal{S}^1\), where $V$ is a one-dimensional smooth holomorphic zero variety in an open neighborhood of $\overline{\mathbb{B}_d}$ and transversal to the boundary $\partial \mathbb{B}_d$; moreover, this trace-class property holds uniformly in the dimension \(d\). Inspired by this, we propose the following conjecture.
\begin{Conj}Let $V$ be a one-dimensional smooth holomorphic zero variety of an open neighborhood of $\overline{\mathbb B}$, which is transversal to the boundary $\partial \mathbb B$. Then the quotient modules defined by $V$ of the Drury-Arveson module in infinitely many variables are essentially normal.
\end{Conj}

 \noindent{Penghui Wang, School of Mathematics, Shandong University, Jinan 250100, Shandong, P. R. China, Email:
phwang@sdu.edu.cn}

\noindent{Ruoyu Zhang, School of Mathematics, Shandong University, Jinan 250100, Shandong, P. R. China, Email:
ryzhangmath@163.com
}

 \noindent{Zeyou Zhu, School of Mathematics, Shandong University, Jinan 250100, Shandong, P. R. China, Email:
13155486329@163.com}


\begin{thebibliography}{99}

\bibitem{ACD} O. P. Agrawal, D. N. Clark and R. G. Douglas, \emph{Invariant subspaces in the polydisk}, Pacific J. Math. 121(1986), no. 1, 1-11.


    \bibitem{Aleman2025} A. Aleman, M. Hartz, J. E. M\textsuperscript{c}Carthy and S. Richter, \emph{Multiplier varieties and multiplier algebras of CNP Dirichlet series kernels}. arXiv:2507.21537 (2025).
        

\bibitem{AglerMcCarthy2000} J. Agler and J. E. M\textsuperscript{c}Carthy, \emph{Complete Nevanlinna-Pick kernels}. J. Funct. Anal. 175 (2000), no. 1, 111-124.
\bibitem{AglerMcCarthyBook2002} J. Agler and J. E. M\textsuperscript{c}Carthy, \emph{Pick interpolation and Hilbert function spaces}. Graduate Studies in Mathematics, 44. American Mathematical Society, Providence, RI, 2002.

    
\bibitem{Atiyah} M. Atiyah and I. Macdonald, \emph{Introduction to commutative algebra}. Addison-Wesley Publishing Co., Reading, Mass.-London-Don Mills, Ont., 1969.


\bibitem{interpolating sequence} A. Aleman, M. Hartz, J. E. M\textsuperscript{c}Carthy and S. Richter, \emph{Interpolating Sequences in Spaces with the Complete Pick Property}. Int. Math. Res. Not. IMRN 2019, no. 12, 3832–3854.

    \bibitem{Arveson curvature} W. Arveson, \emph{The curvature invariant of a Hilbert module over $\mathbb{C}[z_1,\ldots,z_d].$} J. Reine Angew. Math. 522 (2000), 173-236.
    \bibitem{Arveson Dirac} W. Arveson,  \emph{The Dirac operator of a commuting $d$-tuple}.
J. Funct. Anal. 189 (2002), no. 1, 53-79.

\bibitem{Arv1} W. Arveson, \emph{Quotients of standard Hilbert modules.}  Trans. Amer. Math. Soc. 359 (2007), no. 12, 6027–6055.

   \bibitem{CNP Curvature} T. Bhattacharyya and A. Jindal, \emph{Complete Nevanlinna-Pick kernels and the curvature invariant}. Ann. Mat. Pura Appl. (4) 204 (2025), no. 3, 1183–1197.
   \bibitem{CNP character} T. Bhattacharyya and A. Jindal, \emph{Complete Nevanlinna-Pick kernels and the characteristic function}. Adv. Math. 426 (2023), Paper No. 109089, 25 pp.
   \bibitem{CG} X. Chen and K. Guo, \emph{Analytic Hilbert modules}. Chapman {\&} Hall/CRC Research Notes in Mathematics, 433. Chapman \& Hall/CRC, Boca Raton, FL, 2003.
  
  \bibitem{CH}R.~Clou\^atre and M.~Hartz, \emph{Multiplier algebras of complete Nevanlinna-Pick spaces: dilations, boundary representations and hyperrigidity}. J. Funct. Anal. 274 (2018), no. 6, 1690-1738.
  
  \bibitem{CT}R.~Clou\^atre and E.~J.~Timko, \emph{Gelfand transforms and boundary representations of complete Nevanlinna-Pick quotients}. Trans. Amer. Math. Soc. \textbf{373} (2020), no. 2, 1339-1374.
  
  \bibitem{DGW} R. G. Douglas, K. Guo and Y. Wang, \emph{On the p-essential normality of principal submodules of the Bergman module on strongly pseudoconvex domains}. Adv. Math. 407 (2022), Paper No. 108546, 41 pp.
      
  \bibitem{DHS} K.~R.~Davidson, M.~Hartz, and O.~M.~Shalit, \emph{Multipliers of embedded discs}. Complex Anal. Oper. Theory 9 (2015), no.~2, 287-321.
  
   \bibitem{DP} R. G. Douglas and V. Paulsen, \emph{Hilbert modules over function algebras}. Pitman Research Notes in Mathematics Series, 217. Longman Scientific \& Technical, Harlow; copublished in the United States with John Wiley \& Sons, Inc., New York, 1989.
   
   
   \bibitem{DPSY} R. G. Douglas, V. Paulsen, C. Sah and K. Yan, \emph{Algebraic reduction and rigidity for Hilbert modules.} Amer. J. Math. 117 (1995), no. 1, 75–92.
   
   \bibitem{DW} R. G. Douglas and K. Wang, \emph{ A harmonic analysis approach to essential normality of principal submodules.}  J. Funct. Anal. 261 (2011), no. 11, 3155–3180.
   
   \bibitem{DWy} R. G. Douglas and Y. Wang, \emph{Geometric Arveson-Douglas conjecture and holomorphic extensions.} Indiana Univ. Math. J. 66 (2017), no. 5, 1499–1535. 
   
   \bibitem{DY} R. G. Douglas and K. Yan, \emph{On the rigidity of Hardy submodules.}  Integral Equations Operator Theory 13 (1990), no. 3, 350–363.
   
   \bibitem{Eng} M. Engli$\check{\text{s}}$ and J. Eschmeier, \emph{Geometric Arveson-Douglas conjecture.}  Adv. Math. 274 (2015), 606–630.
   
   \bibitem{Fang Xiang} X. Fang, \emph{Hilbert polynomials and Arveson's curvature invariant}. J. Funct. Anal. 198 (2003), no. 2, 445-464.
   
   \bibitem{FX} Q. Fang and J. Xia, \emph{Essential normality of polynomial-generated submodules: Hardy space and beyond. }  J. Funct. Anal. 265 (2013), no. 12, 2991–3008. 
   
   \bibitem{GRS}J. Gleason, S. Richter and C. Sundberg,
 \emph{On the index of invariant subspaces in spaces of analytic functions of several complex variables}.
J. Reine Angew. Math. 587 (2005), 49-76.

\bibitem{GC} K. Guo and X. Chen, \emph{Analytic Hilbert modules.} Chapman  Hall/CRC Research Notes in Mathematics, 433. Chapman  Hall/CRC, Boca Raton, FL, 2003.
\bibitem{GR} H.~Grauert and R.~Remmert, \emph{Coherent Analytic Sheaves}. Grundlehren der mathematischen Wissenschaften [Fundamental Principles of Mathematical Sciences], 265. Springer-Verlag, Berlin, 1984.
\bibitem{Guo1} K. Guo, \emph{Characteristic spaces and rigidity for analytic Hilbert modules.}  J. Funct. Anal. 163 (1999), no. 1, 133–151. 

\bibitem{Guo2} K. Guo, \emph{Equivalence of Hardy submodules generated by polynomials.}  J. Funct. Anal. 178 (2000), no. 2, 343–371. 

   \bibitem{Gu} K. Guo, \emph{Defect operators for submodules of $H_d^2$.} J. Reine Angew. Math. 573 (2004), 181-209.
   \bibitem{Gu2} K. Guo, \emph{Defect operators, defect functions and defect indices for analytic submodules.} J. Funct. Anal. 213 (2004), no. 2, 380-411.
   
   
   \bibitem{inner multipliers}  D. Greene, S. Richter and C. Sundberg, \emph{The structure of inner multipliers on spaces with complete Nevanlinna-Pick kernels.}
J. Funct. Anal. 194 (2002), no. 2, 311-331.

\bibitem{GW} K. Guo and K. Wang, \emph{Essentially normal Hilbert modules and K-homology.} Math. Ann. 340 (2008), no. 4, 907–934. 

\bibitem{GWy}K. Guo and Y. Wang, \emph{A survey on the Arveson-Douglas conjecture.} Operator theory, operator algebras and their interactions with geometry and topology—Ronald G. Douglas memorial volume, 289–311, Oper. Theory Adv. Appl., 278, Birkhäuser/Springer, Cham, 2020.

\bibitem{HM} M. Hartz, \emph{On the Isomorphism Problem for Multiplier Algebras of Nevanlinna-Pick Spaces}. Canad. J. Math. 69 (2017), no. 1, 54–106.
    
\bibitem{Ken} M. Kennedy, \emph{Essential normality and the decomposability of homogeneous submodules.} Trans. Amer. Math. Soc. 367 (2015), no. 1, 293–311.

\bibitem{KS} M. Kennedy and O. Shalit, \emph{Essential normality, essential norms and hyperrigidity.} J. Funct. Anal. 268 (2015), no. 10, 2990–3016.

\bibitem{Levy}R. Levy, \emph{Note on the curvature and index of an almost unitary contraction operator}. Integral Equations Operator Theory 48 (2004), no. 4, 553-555.
   
\bibitem{trent}S. McCullough and T. Trent, \emph{Invariant Subspaces and Nevanlinna–Pick Kernels}. J. Funct. Anal. 178 (2000), no. 1, 226–249.

\bibitem{Muhly}P. Muhly and B. Solel, \emph{The curvature and index of completely positive maps}. Proc. London Math. Soc. (3) 87 (2003), no. 3, 748-778.

\bibitem{Parrott}S. Parrott, \emph{The curvature of a single contraction operator on a Hilbert space}. Preprint, 2000, arXiv:math.OA/0006224 v1.

\bibitem{Popescu2001} G. Popescu, \emph{Curvature invariant for Hilbert modules over free semigroup algebras}. Adv. Math. 158 (2001), no. 2, 264-309.

\bibitem{Popescu2015} G. Popescu, \emph{Curvature invariant on noncommutative polyballs}. Adv. Math. 279 (2015), 104-158.

\bibitem{Pu} M. Putinar, \emph{On the rigidity of Bergman submodules.} Amer. J. Math. 116(1994), no. 6, 1421-1432.

\bibitem{Popescu2017} G. Popescu, \emph{Euler characteristic on noncommutative polyballs}. J. Reine Angew. Math. 2017 (2017), no. 728, 195-236.
    
\bibitem{Ric} S. Richter, \emph{Unitary equivalence of invariant subspaces of Bergman and Dirichlet spaces.} Pacific J. Math. 133(1988), no. 1, 151-156.    

\bibitem{Rotman} J. Rotman, \emph{Advanced Modern Algebra.}  Second edition. Graduate Studies in Mathematics, 114. American Mathematical Society, Providence, RI, 2010.

\bibitem{curv} P. Wang, R. Zhang and Z. Zhu, \emph{Arveson's version of the Gauss-Bonnet-Chern formula for Hilbert modules over the polynomial rings}, J. Reine Angew. Math. 2025 (2025), no. 827, 151-190.
    
    \bibitem{Xia} J. Xia, \emph{Geometric Arveson-Douglas Conjecture for the Drury-Arveson Space: The Case of One-Dimensional Variety}. Adv. Math. 440 (2024), Paper No. 109525, 65 pp.
\end{thebibliography}
\end{document}